\documentclass[12pt]{amsart}
\usepackage{hyperref}
\usepackage{amsmath}
\usepackage{cleveref}
\usepackage{array}   
\newcolumntype{L}{>{$}l<{$}}
\usepackage{amsthm}
\usepackage{mathrsfs}
\usepackage{amssymb}
\usepackage{tikz-cd}
\usepackage{comment}
\usepackage{mathtools}
\usepackage{manfnt}
\calclayout
\usepackage{xcolor}
\usepackage{enumitem}

\newcommand{\Pic}{\mathrm{Pic}}

\newcommand{\al}{\alpha}
\makeatletter
\newcommand{\oset}[3][0ex]{%
  \mathrel{\mathop{#3}\limits^{
    \vbox to#1{\kern-2\ex@
    \hbox{\(\scriptstyle#2\)}\vss}}}}
\makeatother
\newcommand{\posmod}{\oset{+}{\rightarrow}} 
\newcommand{\negmod}{\oset{-}{\rightarrow}}

\newcommand{\PP}{\mathbb{P}}

\newcommand{\ZZ}{\mathbb Z}

\newcommand{\codim}{\mathrm{codim}}

\newcommand{\OO}{\mathcal O}

\newcommand*{\shom}{\mathcal{H}\kern -.5pt om} 
\newcommand{\rk}{\mathrm{rk}}
\newcommand{\Hom}{\mathrm{Hom}}

\newcommand{\ra}{\rightarrow}

\newcommand{\Ext}{\mathrm{Ext}}

\newcommand{\leqor}{\underset{{\scriptscriptstyle (}-{\scriptscriptstyle )}}{<}}

\newcommand{\twopartdef}[4]
{
	\left\{
		\begin{array}{ll}
			#1 & \mbox{if } #2 \\
			#3 &  #4
		\end{array}
	\right.
}
\newtheorem{theorem}{Theorem}[section]

\newtheorem{lemma}[theorem]{Lemma}
\newtheorem{proposition}[theorem]{Proposition}
\newtheorem{corollary}[theorem]{Corollary}

\theoremstyle{definition}

\newtheorem{remark}[theorem]{Remark}

\newtheorem*{example}{Example}

\newtheorem*{theorem*}{Theorem}
\newtheorem{definition}[theorem]{Definition}

\title{Stability of normal bundles of Brill-Noether curves}
\author{Izzet Coskun}
\address{Department of Mathematics, Statistics, and CS \\
University of Illinois at Chicago, Chicago IL 60607}
\email{icoskun@uic.edu}
\author{Geoffrey Smith}
\email{geoff.deg.smith@gmail.com}

\thanks{During the preparation of this article, Izzet Coskun was partially supported
by NSF FRG grant DMS-1664296 and NSF grant DMS-2200684, Geoffrey Smith was supported by an AMS-Simons Travel grant.}

\begin{document}

\begin{abstract}
    We prove that the normal bundle of a general Brill-Noether curve of genus $g \geq 1$ and degree $d$ in $\PP^r$ is semistable if $g=1$ or $
    g\geq \left \lceil \frac{5r}{2}\right\rceil r(r-1),
    $ or $d$ is larger than an explicit function of $g$ and $r$. We further prove that the normal bundle is in fact stable if $g\geq 2$ and either $g$ or $d$ satisfy slightly stronger bounds. In particular, for each $r$, there are at most finitely many $(d,g)$ with $g \geq 1$ (respectively, $g \geq 2$) for which the normal bundle of the general Brill-Noether curve in $\PP^r$ is not semistable (respectively, stable).   
\end{abstract}

\maketitle
\section{Introduction and statement of results}\label{sec-intro}

 In this paper, we study the (semi)stability of the normal bundle $N_{C|\PP^r}$ of a general smooth, projective curve $C$ of genus $g$ embedded in $\PP^r$ by a general nondegenerate map of degree $d$. We work over an algebraically closed field $k$ of arbitrary characteristic. The normal bundle controls the deformations of $C$ in $\PP^r$ and is a fundamental invariant of the embedded curve. It carries crucial information about the local structure of the Hilbert scheme and plays a central role in questions of arithmetic and moduli.

Recall that the slope of a vector bundle $V$ on a curve $C$ is defined by $\mu(V)= \frac{\deg(V)}{\rk(V)}.$ The bundle is called {\em (semi)stable} if every proper subbundle $W$ satisfies $$\mu(W) \leqor \mu(V).$$  Stable bundles satisfy nice cohomological and metric properties, and are the building blocks of all bundles via the Harder-Narasimhan and Jordan-H\"{o}lder filtrations. Consequently, it is important to know when naturally defined bundles (such as $N_{C|\PP^r}$) are (semi)stable. 

By the Brill-Noether Theorem \cite{griffithsharris, jensenpayne, osserman}, a general smooth projective curve $C$ of genus $g$ admits a nondegenerate map of degree $d$ to $\PP^r$ if and only if the Brill-Noether number $\rho(d,g,r)$
satisfies 
$$\rho(d,g,r):= g-(r+1)(g-d+r) \geq 0.$$ We call a triple of nonnegative integers $(d,g,r)$ a {\em Brill-Noether triple} or {\em BN triple} if $\rho(d,g,r)\geq 0$. When $r\geq 3$, the general such map is an embedding and there is a unique component of the Hilbert scheme that dominates the moduli space $\overline{M}_g$ and whose general member parameterizes nondegenerate, smooth genus $g$ curves of degree $d$ in $\PP^r$ \cite{Eishar}. A member of this component is called a {\em Brill-Noether} curve or {\em BN-curve} for short.

\subsection{Results} Define $$u := \left\lfloor \frac{(r+1)^2}{2(r-1)} \right \rfloor.$$ Let $k$ be the integer such that $(k-1)u < g-1 \leq k u$. Our two main theorems are the following.

\begin{theorem}\label{largeMain}
Let $C\subset \PP^r$ be a general BN-curve of degree $d$ and genus $g \geq 1$. Then $N_{C\vert \PP^r}$ is semistable if  one of the following holds:
\begin{enumerate}
    \item $g=1$,
    \item  $
    g\geq \binom{r-1}{2}+2+\left \lceil \frac{5r^2-7r}{2(r-1)}\right\rceil r(r-1),
    $ 
\item $d\geq \min\left(g+\frac{r^2}{4}+2r-3, (k+1)(r+1), (g-1)(2r-3)+r+1\right)$

\end{enumerate}
\end{theorem}

Let $b_2(r)$ be the least integer such that for all $d \geq b_2(r)$, there exist positive integers $d_1, d_2$ such that $d= d_1 + d_2$, $d_1, d_2 \geq r+1$ and $\gcd(r-1, 2d_1 + 1) =1$. 

\begin{theorem}\label{thm-introstab}
   Let $C\subset \PP^r$ be a general BN-curve of degree $d$ and genus $g \geq 2$.  Then $N_{C\vert \PP^r}$ is stable if  one of the following holds:
\begin{enumerate}
    \item  $g \geq \binom{r-1}{2}+3+\left \lceil\frac{2(r+1)b_2(r)+3r^2-13r-2}{2(r-1)}\right \rceil r(r-1),$
\item $d\geq b_2(r)+ \min\left(g+\frac{r^2}{4}+r-5, k(r+1), (g-1)(2r-3)\right)$.
\end{enumerate} 
\end{theorem}

\begin{remark}\label{rem-intro1}
If $p \geq 5$ is the smallest prime number  that does not divide $r-1$, then $b_2(r) \leq 2r+\frac{p-1}{2}$.  In general, $b_2(r)= 2r + O(\log(r))$ and  $b_2(r) \leq \frac{5r-3}{2}$ (see Proposition \ref{b2Bounds}). In fact, if $r \geq 1636$, then $b_2(r) \leq 2.01r + 2.015$ (see Remark \ref{rem-be}).
\end{remark}

\begin{remark}\label{rem-intro2}
At the expense of slightly weakening the bounds, one can replace the expression $d \geq (k+1)(r+1)$ in Theorem \ref{largeMain} with $d \geq \frac{2(r^2-1)}{r^2+3}(g-1) + 2r+2$ to obtain an expression purely in terms of $g$ and $r$. Similarly, one can replace the expression $d \geq b_2(r) + k(r+1)$ in Theorem \ref{thm-introstab} with $d \geq \frac{2(r^2-1)}{r^2+3}(g-1) + b_2(r) + r+1.$ 
\end{remark}

\begin{remark}
   Theorems \ref{largeMain} and \ref{thm-introstab} imply that for each $r$, there are finitely many pairs $(d,g)$ with $g \geq 1$ (respectively, $g\geq 2$) for which the normal bundle of the general BN-curve in $\PP^r$ is not semistable (respectively, stable).
\end{remark}

\begin{remark}
Coskun, Larson and Vogt \cite[Conjecture 1.1]{clv22} conjectured that there are finitely many triples $(d,g,r)$ with $g \geq 2$ for which the normal bundle of a general BN-curve is not stable. Our results provide strong evidence for this conjecture and reduce the problem to studying the normal bundles of certain low degree and low genus curves in $\PP^r$. 
\end{remark}

We have sharper results in the specific case $r=4$. First, using a non-degenerative cohomological method, we prove the following.
\begin{theorem}[Theorem \ref{724}]
If $C\subset \PP^4$ is a general BN-curve of degree 7 and genus 2, and the ground field does not have characteristic 2, then $N_{C\vert \PP^4}$ is stable. 
\end{theorem}
Using this theorem and more precise versions of the numerics used to prove Theorem \ref{largeMain}, we prove the following.
\begin{theorem}\label{theorem-P4}
Assume the ground field has characteristic different from $2$. Let $(d,g)$ be a pair of positive integers with $\rho(d,g,4)\geq 0$ and $(d,g)\neq (6,2)$. If $(d,g)$ is not one of the 48 values listed in Table \ref{table-P4ss},  then a general BN-curve of degree $d$ genus $g$ in $\PP^4$ has semistable normal bundle. Furthermore, if $g \geq 2$, $(d,g)$ is not $(8,5)$ or one of the 63 additional values listed in Table \ref{table-P4s}, then a general BN-curve of degree $d$ genus $g$ in $\PP^4$ has stable normal bundle.
\end{theorem}

\begin{remark}
Our proof of Theorem \ref{724} requires the ground field to have characteristic different from  $2$. When the characteristic is $2$, we do not know the semistability of $N_{C|\PP^4}$  when $(d,g)\in \{(7,2),(16,14)\}$ in addition to the cases listed in Table \ref{table-P4ss}. 
\end{remark}

\subsection{Unstable normal bundles}\label{sec-introunstable} We call a  BN triple $(d,g,r)$ a {\em (semi)stable triple} if the general BN-curve of degree $d$ and genus $g$ in $\PP^r$ has a (semi)stable normal bundle. Otherwise, we call $(d,g,r)$ an {\em unstable} triple. There are some examples of unstable triples.

\begin{example}[Canonical curves] The triples $(6,4,3)$ and $(10,6,5)$ corresponding to canonical curves of genus $4$ and $6$ are unstable triples (see also \cite{clv23}). A canonical curve $C$ of genus $4$ is a $(2,3)$ complete intersection in $\PP^3$ and the normal bundle of the curve in the quadric destabilizes $N_{C|\PP^3}$. A general canonical curve $C$ of genus $6$ lies in a del Pezzo surface $X$ of degree $5$ and $N_{C|X}$ destabilizes $N_{C|\PP^5}$.

The triple $(8,5,4)$ corresponding to a canonical curve of genus $5$ is semistable but not stable. The general canonical curve $C$ of genus 5 is a $(2,2,2)$ complete intersection in $\PP^5$, hence $N_{C|\PP^5}= \OO_C(2)^{\oplus 3}$ and is semistable but not stable.
\end{example}

\begin{example}[Genus $2$ curves] The triples $(5,2,3)$, $(6,2,4)$, $(7,2,5)$ and $(8,2,6)$ are unstable. Every genus $2$ curve $C$ is hyperelliptic and lies in a rational surface scroll $X$ swept out by lines spanned by fibers of the $g_2^1$. The normal bundle $N_{C|X}$ is a line bundle of degree 12 which destabilizes $N_{C|\PP^r}$ for these 4 triples.
\end{example}

\begin{remark}
  There may be special BN-curves whose normal bundles are unstable even when $(d,g,r)$ is a (semi)stable BN triple. For example, normal bundles of trigonal or tetragonal canonical curves of $g=5$ or $g \geq 7$ are unstable even though the general canonical curve has semistable normal bundle \cite{AFO, clv23}. In general, given an arbitrary BN-curve $C \subset \PP^r$, determining the Harder-Narasimhan filtration of $N_{C|\PP^r}$ is an interesting question, especially when $C$ is a multi-canonical curve.  
\end{remark}

\subsection{History of the problem} Normal bundles of curves in projective space have been studied extensively (see for example \cite{aly, lv22} and references therein). 

If $C \cong \PP^1$, then the normal bundle is a direct sum of line bundles $N_{C/\PP^r}= \oplus_{i=1}^{r-1} \OO_{\PP^1}(a_i)$.  In characteristic distinct from 2, the normal bundle of the general rational curve is balanced \cite{ghionesacchiero, sacchiero80, ran} and the loci with a fixed splitting type has been studied \cite{sacchiero82, eisenbudvandeven, eisenbudvandeven82, coskunriedl}.  In characteristic 2, the splitting type of the normal bundle satisfies a parity condition $a_i \equiv d \mod{2}$ and the splitting type of the normal bundle of the general rational curve is as balanced as possible subject to this parity condition (see \cite{clv22, lv22} and \cite[Proposition 4.3, Remark 4.4]{clv24}). 

Normal bundles of genus one curves have been studied by \cite{einlazarsfeld, ellingsrudhirschowitz, ellingsrudlaksov}. Ein and Lazarsfeld \cite{einlazarsfeld} prove that the normal bundle of an  elliptic normal curve is semistable.

In the 1980s several authors investigated the stability of normal bundles of BN-curves in $\PP^3$ in low degree and genus.  The stability was proved for $(d, g) = (6, 2)$ by Sacchiero \cite{sacchiero}, for $(d, g) = (9, 9)$ by Newstead \cite{newstead}, for $(d, g) = (6, 3)$ by Ellia \cite{ellia}, and for $(d, g) = (7, 5)$ by Ballico and Ellia \cite{ballicoellia}.  Ellingsrud and Hirschowitz \cite{ellingsrudhirschowitz} announced a proof of stability of normal bundles of space curves in an asymptotic range of degrees and genera. Finally, Coskun, Larson and Vogt in \cite{clv22} proved that the normal bundle of a general BN space curve of genus $g \geq 2$ is stable except when $(d,g) \in \{ (5,2), (6,4)\}$. Consequently, in the rest of the article we will concentrate on the case $r \geq 4$. Forthcoming work of Coskun, Jovinelly and Larson \cite{cjl24} will analyze the cases in $\PP^4$ not covered in this paper.

The stability of the normal bundle of a general canonical curve has also received considerable attention.  Aprodu, Farkas and Ortega \cite{AFO} conjectured that the normal bundle of a general canonical curve of genus at least 7 is stable and settled the case $g=7$. Bruns \cite{Bruns} proved the stability when $g=8$. Finally, Coskun, Larson and Vogt \cite{clv23} proved that the normal bundle of a general canonical curve of genus $g \geq 7$ is semistable.

Larson and Vogt \cite{lv22} solved the interpolation problem for normal bundles of BN-curves. When the rank divides the degree interpolation implies semistability. Hence, Larson and Vogt proved the semistability of normal bundles of BN-curves in a large number of cases. 

Ballico and Ramella \cite{br99} prove that if $g \geq 2r$ and $d \geq 2g+3r+1$, then the general BN curve has a stable normal bundle.  Recently, Ran obtained different asymptotic (semi)stability results for the normal bundles of BN-curves in $\PP^r$ and more generally  normal bundles of curves in hypersurfaces \cite{ran23}. Our results significantly improve the earlier results for $\PP^r$. In particular,  we prove that for each $r$  there are at most finitely many BN triples $(d,g,r)$ which may fail to be (semi)stable. The semistability of the normal bundles of curves has already found applications to Ulrich bundles (see \cite{aclr}).

\subsection{The strategy} To prove Theorems \ref{largeMain} and \ref{thm-introstab}, we first study the case  $g=1$. By the result of Ein and Lazarsfeld \cite{einlazarsfeld}, the normal bundle of an elliptic normal curve is semistable. We show the same for the general elliptic curve using descending induction on $r$. 

We then induct on the degree $d$ and the genus $g$ by specializing $C$ to a reducible BN-curve where some of the components are rational and elliptic curves meeting the rest of the curve in a specified number of points. The degenerations we use are outlined in Section \ref{degenerationSection}. The case of $(7,2,4)$ requires a more detailed and subtle analysis which we carry out in \S \ref{724Section}.

\subsection*{Organization of the paper} We recall basic facts about normal bundles of nodal curves and limits of BN-curves in \S \ref{sec-prelim}.  In \S \ref{sec-elliptic}, we study the case of elliptic curves. In  \S \ref{sec-degenerations}  we carry out our inductive procedure. We prove Theorem \ref{largeMain} and give more precise bounds for BN-curves in $\PP^4$ in \S \ref{largeSection}. Finally, 
in \S \ref{724Section}, we study the case $(7,2,4)$. 

\subsection*{Acknowledgements} We thank Eric Jovinelly, Eric Larson and Isabel Vogt for many discussions about normal bundles of BN-curves.  We thank Angelo Lopez for pointing out the reference \cite{br99}. We are especially grateful to Eric Larson who provided the arguments for Proposition \ref{prop-g1distinct}, corrected an earlier error in our exposition  and helped   significantly improve the paper. We thank the referees for their careful reading of the paper and their numerous suggestions and corrections.

\section{Preliminaries}\label{sec-prelim}
In this section, we recall some preliminaries concerning normal bundles of nodal curves.

\subsection{Normal bundles} Let $C \subset \PP^r$ be a smooth, projective curve of degree $d$ and genus $g$. The exact sequence
$$0 \to TC \to T\PP^r|_C \to N_{C|\PP^r} \to 0$$
implies that $\deg(N_{C|\PP^r}) = d(r+1) + 2g -2$. Since the rank of $N_{C|\PP^r}$ is $r-1$, the slope is given by  $$\mu(N_{C|\PP^r}) = d +  \frac{2(d+g-1)}{r-1}.$$

\subsection{Elementary modifications and normal bundles of nodal curves} Let $X$ be a projective variety and let $V$ be a vector bundle on $X$. Let $D\subset X$ be a Cartier divisor and let $W \subset V|_D$ be a subbundle of the restriction of $V$ to $D$. Then {\em the negative elementary modification} $V[D \stackrel{-}{\to}W]$ of $V$ along $D$ towards  $W$ is defined by the exact sequence
$$0 \to V[D \stackrel{-}{\to}W] \to V \to V|_D/W \to 0.$$ The {\em positive elementary modification} of $V$ along $D$ towards $W$ is defined by 
$$V[D \stackrel{+}{\to}W]:=V[D \stackrel{-}{\to}W](D).$$ We refer the reader to \cite[\S 2-6]{aly} for a detailed discussion of the properties of modifications. The  modification $V[D \stackrel{+}{\to}W]$ is naturally isomorphic to $V$ outside of $D$, so one can easily define multiple modifications of $V$ along divisors with disjoint support, which we will denote by $V[D_1 \stackrel{+}{\to}W_1] \cdots [D_j\stackrel{+}{\to}W_j]$. When the supports of the divisors meet, then we will assume that the bundle $W_i$ extends to  a bundle in a neighborhood of $D_i$, so that we can define multiple modifications. We refer the reader to \cite{aly, lv22} for details about multiple modifications and their properties.

When $C= X \cup Y \subset \PP^r$ be a connected nodal curve. Then at a node $p_i$ of $C$, $T_{p_i}Y$ defines a normal direction in $N_{X|\PP^r}|_{p_i}$. We write $N_{X|\PP^r}[p_i\stackrel{+}{\to}Y]$ for the positive elementary modification $N_{X|\PP^r}[p_i\stackrel{+}{\to}T_{p_i}Y]$. We then have the following fundamental observation of Hartshorne and Hirschowitz.

\begin{lemma}\cite[Corollary 3.2]{hartshornehirschowitz}
Let $C=X\cup Y \in \PP^r$ be a connected nodal curve. Let $X \cap Y= \{ p_1, \dots, p_j\}$. Then 
$$N_{C|\PP^r}|_X \cong N_{X|\PP^r}[p_1 \stackrel{+}{\to} Y ] \cdots [p_j \stackrel{+}{\to} Y ]$$
\end{lemma}

\subsection{Stability on nodal curves} We will specialize our curves to reducible nodal curves. We briefly recall an extension of the notion of stability developed in \cite[\S 2]{clv22}. Let \(C\) be a connected nodal curve and let $\nu \colon \tilde{C} \to C$  be its normalization.  For any node \(p\) of \(C\), let \(\tilde{p}_1\) and \(\tilde{p}_2\) be the two points of \(\tilde{C}\) over \(p\). Given a vector bundle \(V\) on \(C\), the fibers of the pullback \(\nu^*V\) to \(\tilde{C}\) over \(\tilde{p}_1\) and \(\tilde{p}_2\) are naturally identified.  Given a subbundle \(F \subseteq \nu^*V\), we can thus compare \(F|_{\tilde{p}_1}\) and \(F|_{\tilde{p}_2}\) inside \(\nu^*V|_{\tilde{p}_1} \simeq \nu^*V|_{\tilde{p}_2}\).

\begin{definition}
Let \(V\) be a vector bundle on a connected nodal curve \(C\).  
For a subbundle \(F \subset \nu^*V\),  the {\em adjusted slope \(\mu^{\text{adj}}_C\)} is given by  
\[\mu^{\text{adj}}_C(F) := \mu(F) - \frac{1}{\rk{F}} \sum_{p \in C_{\text{sing}}} \codim_{F} \left(F|_{\tilde{p}_1}\cap F|_{\tilde{p}_2} \right),\]
where \(\codim_F \left(F|_{\tilde{p}_1}\cap F|_{\tilde{p}_2} \right)\) refers to the codimension
of the intersection in either \(F|_{\tilde{p}_1}\) or \(F|_{\tilde{p}_2}\). 
Then \(V\) is {\em (semi)stable} if for all subbundles \(F \subset \nu^*V\),
\[\mu^{\text{adj}}(F) \leqor \mu(\nu^*V) = \mu(V). \]
\end{definition}

\noindent
This notion of semistability specializes well in families of nodal curves. The following two results make our inductive approach to the problem possible.

\begin{proposition} \cite[Proposition 2.3]{clv22} \label{prop:stab-open}
Let \(\mathcal{C} \to \Delta\) be a family of connected nodal curves over
the spectrum of a discrete valuation ring,
and \(\mathcal{V}\) be a vector bundle on \(\mathcal{C}\).
If the special fiber \(\mathcal{V}_0 = \mathcal{V}|_0\) is (semi)stable,
then the general fiber \(\mathcal{V}^* = \mathcal{V}|_{\Delta^*}\) is also (semi)stable.
\end{proposition}

\begin{lemma} \cite[Lemma 4.1]{clv22}\label{lem:naive}
Suppose that \(C  = X \cup Y\) is a reducible nodal curve and \(V\) is a vector bundle on \(C\) such that \(V|_X\) and \(V|_Y\) are semistable.
Then \(V\) is semistable.
Furthermore, if one of \(V|_X\) or \(V|_Y\) is stable, then \(V\) is stable.
\end{lemma}

\subsection{Specializations of BN-curves}\label{degenerationSection}
We will specialize the general BN-curve to unions $C\cup R$, where $C$ is a general BN-curve of smaller degree and/or genus and $R$ is a genus 0 or 1 curve intersecting $C$ in some fixed number of points. In this subsection, we verify that these are BN-curves.  Recall that two curves in $\PP^r$ intersect {\em quasitransversely} if at each  point of intersection they are locally contained in a smooth surface in which they intersect transversely. In particular, the union of two nodal curves intersecting quasitransversely is a nodal curve. 

\begin{lemma}\label{reducibleBNCurves}
Let $C$ be a BN-curve of type $(d,g,r)$. Then the general nodal union of $C$ and a rational normal curve $R$ of degree $e$  intersecting $C$ quasitransversely in $\alpha$ points is BN if one of the following holds:
\begin{enumerate}
\item $\alpha< e+2$, or
\item $e=r$ and $\alpha=r+2$, or 
\item  $e= r-1$, $\alpha=r+1$, $d \geq r+1$ and $\rho(d,g,r)\geq 1$.
\end{enumerate}
\end{lemma}
\begin{proof}
As in \cite[Lemma 2.10]{clv22}, by basic deformation theory, it suffices to show $H^1(T\PP^r\vert_{C\cup R})=0$. Consider the exact sequence
\[
0\ra T\PP^r\vert_R(-R\cap C)\ra T\PP^r\vert_{C\cup R}\ra T\PP^r\vert_C\ra 0.
\]
Since $C$ is a general BN-curve, by the Gieseker-Petri Theorem $H^1(C,T\PP^r\vert_C)=0$. Moreover, $T\PP^r\vert_R\cong\OO(e+1)^{e} \oplus \OO(e)^{r-e}$, so if $\al<e+2$ or $e=r$ and $\al=r+2$, we have $H^1(R,T\PP^r\vert_R(-\alpha))=0$. Hence, Cases (1) and (2) hold.

Case (3) is  \cite[Lemma 5.7]{lv22}.
\end{proof}
\begin{proposition}\label{230}
Let $C_1,C_2\subset \PP^r$ be general genus 1 BN-curves that  intersect quasitransversely in a collection of fewer than $\max(\deg(C_1),\deg(C_2))$ nodes $\Gamma$. Then $C=C_1\cup C_2$ is a BN-curve.
\end{proposition}
\begin{proof}
By deformation theory, we need to show $H^1(C, T\PP^r|_C)=0$. Without loss of generality, assume  $\deg(C_1) \geq \deg(C_2)$. 
 We have the exact sequence
\[
0\ra T\PP^r\vert_{C_1}(-\Gamma)\ra T\PP^r\vert_{C} \ra T\PP^r\vert_{C_2}\ra 0.
\]
Since $C_2$ is a general BN-curve, $H^1(C_2, T\PP^r\vert_{C_2})=0$.  By Serre duality,  $$H^1(C_1, T\PP^r\vert_{C_1}(-\Gamma))\cong H^0(C_1, \Omega\PP^r\vert_{C_1}(\Gamma))^*,$$ where $\Omega \PP^r$ denotes the cotangent bundle of $\PP^r$. By restricting the Euler sequence, we have an exact sequence
\[
0\ra \Omega\PP^r\vert_{C_1}\ra \OO_{C_1}(-H)^{r+1}\ra \OO_{C_1}\ra 0.
\]
so since $\Gamma\subset C_1$ has degree less than $\deg(C_1)$, we have that $\Omega \PP^r\vert_{C_1}(\Gamma)$ is a subsheaf of a bundle $\OO_{C_1}(-H+\Gamma)^{r+1}$ with no global sections, so $H^1(C_1, T\PP^r\vert_{C_1}(-\Gamma))=0$. We conclude $H^1(C, T\PP^r|_C)=0$ and the proposition follows.
\end{proof}

The proof in fact shows the following.

\begin{corollary}\label{cor-g1bn}
Let $C_1$ be a general genus 1 BN-curve and let $C_2$ be a BN curve with $H^1(C_2, T\PP^r|C_2)=0$ intersecting $C_1$ quasitransversely at fewer than $\deg(C_1)$ points. Then $C_1 \cup C_2$ is a BN-curve.
\end{corollary}

\subsection{Interpolation and stability} In this subsection, we recall Larson and Vogt's result on interpolation and discuss the relation between interpolation and semistability of the normal bundle. 

\begin{definition}
A vector bundle $E$ on a curve $C$ \emph{satisfies interpolation} if $H^1(C,E)=0$, and, for all $e>0$, there exists an effective divisor $D$ of degree $e$ on $C$ such that $H^0(C,E(-D))=0$ or $H^1(C,E(-D))=0$.
\end{definition}

Larson and Vogt prove the following theorem. 

\begin{theorem}\cite[Theorem 1.2]{lv22}\label{lvMain} The normal bundle of a general BN-curve of degree $d$ and genus $g\geq 1$ satisfies interpolation except when 
$$(d,g,r) \in \{(5,2,3), (6,4,3), (6,2,4), (7,2,5), (10, 6, 5)\}$$
\end{theorem}

They also observe that interpolation implies semistability when the rank of the normal bundle divides the degree. We recall the argument for the reader's convenience. 

\begin{corollary}\label{interpolation}\cite[Remark 1.6]{vogt}
Let $(d,g,r)$ be a BN-triple such that $r-1$ divides $2d+2g-2$. Assume $$(d,g,r) \notin \{(5,2,3), (6,4,3), (7,2,5)\}.$$
Then for the general BN-curve $C$ of type $(d,g,r)$, $N_{C|\PP^r}$ is semistable. 
\end{corollary}
\begin{proof}
If $r-1$ divides $2d+2g-2$, then $\mu(N_{C\vert \PP^{r}})$ is an integer. Let $D$ be a divisor on $C$ of degree $\mu(N_{C\vert \PP^{r}})-g+1$ such that $H^0(C,N_{C\vert \PP^{r}}(-D))=0$ or $H^1(C,N_{C\vert \PP^{r}})=0$. Since $\chi(C,N_{C\vert \PP^{r}})=0$, we must in fact have 
\[
H^0(C,N_{C\vert \PP^{r}}(-D))=H^1(C,N_{C\vert \PP^{r}}(-D))=0.
\]
Let $S$ be a nonzero subbundle of $N_{C|\PP^r}$. Then we have $H^0(S(-D))=0$, so $\chi(S(-D))=(r-1)(\mu(S)-\mu(N_{C\vert \PP^r}))$ is nonpositive and $\mu(S)\leq \mu(N_{C\vert \PP^r})$.
\end{proof}

\begin{remark}
 The exceptional cases $(6,2,4)$ and $(10,6,5)$ do not appear in the corollary because $r-1$ does not divide $2d+2g-2$.  
\end{remark}

\subsection{Pointing bundles} Let $q \in \PP^r$ be a point.  In this subsection, we recall the definition of the pointing bundle $N_{C \to q}$. We refer the reader to \cite[\S 5 and 6]{aly} for  more details. Let $U_{C,q} := \{ p \in C | T_p C \cap q = \emptyset\}$. Let $\pi_q: C \to \PP^{r-1}$ denote the projection from $q$ and observe that $\pi_q$ is unramified on $U_{C,q}$. If $U_{C,q}$ is dense in $C$ and contains the singular locus of $C$, then $N_{C\to q}$ is the unique extension of the kernel of the natural map $N_{C|\PP^r}|_{U_{C,q}} \to N_{\pi_q}|_{U_{C,q}},$ where $N_{\pi_q}$ denotes the normal sheaf of $\pi_q$. In this case, projection from $q$ induces the exact sequence
$$0 \to N_{C\to q} \to N_{C|\PP^r} \to \pi_q^* N_{\pi_q} (C \cap q) \to 0.$$ We will use pointing bundles only in the simplest two cases:
\begin{enumerate}
    \item If $q$ is a general point and $U_{C, q} = C$, then $N_{C\to q} \cong \OO_C(1)$ \cite[Proposition 6.2]{aly}.
    \item If $q$ is a general point of $C$ and $U_{C, q} = C \setminus \{q\}$, then $N_{C\to q} \cong \OO_C(1)(2q)$ \cite[Proposition 6.3]{aly}.
\end{enumerate}
By convention, we set $N_{C|\PP^r}[p \posmod q]$ to be $N_{C|\PP^r}[p \posmod N_{C \to q}]$.

\section{Normal bundles of curves of genus 1}\label{sec-elliptic}
In this section, we prove that  the normal bundle of a general  genus $1$ BN-curve in $\PP^r$ is semistable.

\begin{theorem}\label{genus1}
Let $C\subset \PP^r$ be a general BN-curve of degree $d\geq r+1 \geq 4$ and genus $g=1$. Then $N_{C\vert \PP^r}$ is semistable.
\end{theorem}
\begin{proof}
We fix $d$ and prove the theorem by descending induction on $r$. The base case, $r=d-1$, is the $i=1$ case of \cite[Theorem 4.1]{einlazarsfeld}.

Suppose the theorem is true for degree $d$ genus $1$ BN-curves in $\PP^{r+1}$. Let $\tilde{C}\subset \PP^{r+1}$ be a general degree $d$ genus 1 curve, and let $C\subset \PP^r$ be the projection of $\tilde{C}$ from a general point $p \in \PP^{r+1}$. Suppose $N_{C\vert \PP^r}$ is unstable with a destabilizing quotient $Q$ of minimal slope. 

We first show that $\deg(Q)$ is a multiple of $d$. Let $H$ be the hyperplane class on $C$. Then the complete linear series $|H|$ embeds $C$ in $\PP^{d-1}$ and $C$ is the projection of this embedding by a linear space of dimension $d-r-2$. For an open set $U$ in $\mathbb{G}(d-r-2, d-1)$, projection from $\Lambda \in U$ yields an elliptic curve of degree $d$ in $\PP^r$ whose normal bundle has a destabilizing quotient $Q_{\Lambda}$. Since there are no nonconstant rational maps from a Grassmannian to an abelian variety, $\det(Q_{\Lambda})$ depends only on the abstract curve and the linear system $|H|$. Let $\tau:C\ra C$ be translation by a $d$-torsion point. Since $\tau^*(\OO(H))\cong\OO(H)$, 
 $\det(Q)$ is also invariant under pullback by $\tau$. Hence, $\deg(Q)$ is a multiple of $d$.

We have the following diagram associated to the projection.
 \begin{equation}
     \begin{tikzcd}
         &&S\arrow[d,hookrightarrow]\\
         N_{\tilde{C}\to p}\arrow[r,hookrightarrow]&N_{\tilde{C}\vert \PP^{r+1}}\arrow[r,twoheadrightarrow]&N_{C\vert \PP^r}\arrow[d,twoheadrightarrow]\\
         &&Q
     \end{tikzcd}
 \end{equation}
$$\mbox{with} \quad \rk(N_{\tilde{C}\vert \PP^{r+1}})=r, \quad \deg(N_{\tilde{C}\vert \PP^{r+1}})=d(r+2),$$ $$\rk(N_{C\vert \PP^r})=r-1\quad  \mbox{and} \quad \deg(N_{C\vert \PP^{r}})=d(r+1).$$
Since $N_{\tilde{C}\vert \PP^{r+1}}$ is semistable, the slope of $Q$ satisfies
\[
\mu(N_{\tilde{C}\vert \PP^{r+1}})\leq \mu(Q) < \mu(N_{C\vert \PP^r}),
\]
hence
\[
d+\frac{2d}{r}\leq \mu(Q)< d+\frac{2d}{r-1}.
\]
Suppose $Q$ has rank $s$; since $Q$ has degree a multiple of $d$, we have that $\deg(Q)/d$ is an integer satisfying
\[
s+\frac{2s}{r} \leq \frac{\deg(Q)}{d}< s+\frac{2s}{r-1}.
\]
Since $0<s< r-1$, the only way this can be achieved is if $s=\frac{r}{2}$. If $r$ is odd, this leads to a contradiction immediately. If, instead, $r$ is even, we must have 
\[
\mu(Q)=d+\frac{2d}{r}=\mu(N_{\tilde{C}\vert \PP^{r+1}}).
\]
Set $\mu =d+\frac{2d}{r}$.
We establish a contradiction by showing that no stable bundles of slope $\mu$ can appear as a quotient of $N_{C\vert \PP^r}$. Since $Q$ is a quotient of $N_{\tilde{C}\vert \PP^{r+1}}$, any slope $\mu$ stable quotient of $Q$ appears on the finite list of stable quotient bundles in the Jordan-H\"older filtration of $N_{\tilde{C}\vert \PP^{r+1}}$ of slope $\mu$. We study those bundles now.

Let $a=\gcd(2d,r) = \gcd(\deg(N_{\tilde{C}\vert \PP^{r+1}}), \rk(N_{\tilde{C}\vert \PP^{r+1}}))$. By Atiyah's classification of semistable bundles on elliptic curves \cite{Atiyah},
$N_{\tilde{C}\vert \PP^{r+1}}$ is an iterated extension of $a$ stable bundles of rank $b:=\frac{r}{a}$ and degree $e:=\frac{d(r+2)}{a}$. Call these bundles $E_1,E_2,\ldots,E_a$. 

Suppose $E_1$ appears in this list the maximal number of times, $m$ times, and that $c$ distinct bundles $E_1,\ldots,E_c$ appear exactly $m$ times. Then $$\det(E_1)\det(E_2)\cdots \det(E_c)\in \Pic^{\frac{cd(r+2)}{a}}(C)$$ is determined just by the pair $(C,H)$, and as above $\frac{cd(r+2)}{a}$ must be a multiple of $d$. Since $a|r$,  we conclude that $a$ divides $2c$. Therefore, either  $a=c$ and all of $E_1,\ldots,E_a$ are distinct or $c=a/2$ and every bundle appears exactly twice. 

If each $E_i$ appears exactly once in the list $E_1,\ldots, E_a$, then $\Hom(N_{\tilde{C}\vert \PP^{r+1}}, E_i)$ is one-dimensional for all $i$, spanned by some map $f_i$. Fix $q\in \tilde{C}$, and define $V_q\subset \PP^{r+1}$ as the projectivization of the kernel of the composition
\[
T\PP^{r+1}\vert_{\tilde{C}}\ra N_{\tilde{C}\vert \PP^{r+1}}\xrightarrow{f_i} E_i.
\]
Since $V_q$ has codimension one in $\PP^{r+1}$ and the point of projection $p$ is general, $p \notin V_q$. 
Then the image of $N_{\tilde{C} \to p}\vert_q\ra N_{\tilde{C}\vert \PP^{r+1}}\vert_q$ is not contained in the fiber of $\ker(N_{\tilde{C}\vert \PP^{r+1}}\xrightarrow{f_i} E_i)$ at $q$, so $E_i$ cannot appear in any quotient of $N_{C\vert \PP^r}$.

Now suppose that each $E_i$ appears in the list $E_1,\ldots, E_a$ twice, and fix some $i$. If $\dim(\Hom(N_{C\vert \PP^{r+1}}, E_i))=1$, then the argument in the previous paragraph shows that $E_i$ cannot be a direct summand of the quotient $\frac{N_{\tilde{C}\vert \PP^{r+1}}}{N_{\tilde{C}\to p}}$ for $p$ general. Otherwise, we have $\dim(\Hom(N_{\tilde{C}\vert \PP^{r+1}}, E_i))=2$. In this case, let $f_1,f_2:N_{\tilde{C}\vert \PP^{r+1}}\ra  E_i$ be the two independent maps, and for each $q\in \tilde{C}$ define $V_{j,q}\subset \PP^{r+1}$ as the projectivization of the kernel of the composition
\[
T_q\PP^{r+1}\ra N_{\tilde{C}\vert \PP^{r+1}}\vert_q \xrightarrow{f_j} E_i.
\]
Since $C$ is nondegenerate,  $V_{1,q} \not= V_{2,q}$. Moreover, since $T_q \tilde{C} \subset V_{j,q}$ and $\tilde{C}$ is nondegenerate, $V_{1, q'} \cap V_{1,q} \not= V_{2, q'} \cap V_{1,q}$ for a general $q' \in \tilde{C}$.

 Fix $q\in \tilde{C}$ and let $q'\in \tilde{C}$ be a general point. Let $\ell\subset V_{1,q}$ be a line such that $q \in \ell$, $\ell \not\in V_{2,q}$, $\ell \not\in V_{1, q'}$ and there exists $p'\in \ell$ such that  $p' \in V_{2,q'}\setminus V_{1,q'}$. Then \[
N_{\tilde{C}\to p'}\vert_q\subset N_{\tilde{C}\vert \PP^{r+1}}\vert_q
\]
is contained in the kernel of $f_1$ and not $f_2$, while
\[
N_{\tilde{C}\to p'}\vert_{q'}\subset N_{\tilde{C}\vert \PP^{r+1}}\vert_{q'}
\]
is contained in the kernel of $f_2$ but not $f_1$.
Hence, there is no linear combination of the maps $f_1$ and $f_2$ that is identically zero on $N_{\tilde{C}\to p'}$. Since this is an open property for $p'$, for the general $p$ we chose originally, we have that $E_i$ cannot appear in the quotient $Q$ of $N_{C\vert \PP^{r}}\cong \frac{N_{\tilde{C}\vert \PP^{r+1}}}{N_{\tilde{C}\to p}}$.

Therefore, no $E_i$ can appear as a quotient of $Q$, and hence $Q$ cannot be a quotient of $N_{\tilde{C}\vert \PP^{r+1}}$. We conclude that $N_{C|\PP^r}$ is semistable.
\end{proof}

\begin{remark}
    Theorem \ref{genus1} significantly improves Ran's earlier asymptotic result  asserting  the semistability of the normal bundle of a general genus 1 curve of degree $d = 2r-2$ or $d \geq 3r-3$ \cite[Theorem 2]{ran23}.
\end{remark}

The discussion in the rest of this section was communicated to us by Eric Larson who also corrected an error in our earlier exposition.

\begin{proposition}\label{prop-g1distinct}
Let $r \geq 4$ and let $C \subset \PP^r$ be a general, nondegenerate elliptic curve of degree $d$ where $r-1$ divides $2d$.
\begin{enumerate}
 \item If the characteristic of the base field is 2, and $r$ is odd, and $r-1$ does not divide $d$, then $N_{C|\PP^r}$ is a direct sum of $\frac{r-1}{2}$ indecomposable rank 2 bundles. 
    \item Otherwise, $N_{C|\PP^r}$ is a direct sum of distinct line bundles.
   
\end{enumerate}
\end{proposition}

\begin{proof}[Proof of Proposition \ref{prop-g1distinct}]
    Let $C \subset \PP^r$ be a general nondegenerate curve of degree $d$ and genus 1. Since $C$ is nondegenerate, $d \geq r+1$. Since $r-1$ divides  $2d$, if  $r$ is odd, we have $d= \frac{k(r-1)}{2}$ with $k \geq 3$. If $r$ is even, then $d= m(r-1)$ for $m \geq 2$.

    \subsubsection*{Step 1:} By the argument in the proof of Theorem \ref{genus1}, $N_{C|\PP^r}$ is either a direct sum of $(r-1)$ distinct line bundles or it is a direct sum of $\frac{r-1}{2}$ distinct rank $2$ bundles obtained by a (possibly trivial) extension of the same line bundle.  When $r$ is even, the latter possibility cannot happen and $N_{C|\PP^r}$ must be a direct sum of distinct line bundles. Hence, we may assume that $r$ is odd.

    \subsubsection*{Step 2: Reduce to the case $k=3$ or $4$} We will prove part (2) of the proposition by induction on $k$ and $r$. We first reduce to the case $k=3$ or $4$ by induction on $k$. Specialize $C$ to $C' \cup \cup_{i=1}^s \ell_i$, where $\ell_i$ are $s$ disjoint 1-secant lines meeting $C'$ quasitransversely at the points $p_i$.  By Lemma \ref{lem-line} below, the bundle $N_{C'|\PP^r}[2p_1 \posmod \ell_1] \cdots [2p_s \posmod \ell_s](p_1 + \cdots + p_s)$ is a limit of the normal bundles $N_{C|\PP^r}$. Assume that for some $k$, we have $N_{C|\PP^r} \cong \oplus_{i=1}^{r-1} L_i$ where $L_i$ are distinct line bundles. Pick general points $p_1, \dots, p_{r-1}$ on $C$. At the point $p_j$ attach a line $\ell_j$ quasitransverse to $C$.  Specializing all the points $p_i$ to a general point $p$ while keeping the lines pointing in general directions, specializes $N_{C|\PP^r}[2p_1 \posmod \ell_1] \cdots [2p_{r-1} \posmod \ell_{r-1}] (p_1 + \cdots + p_{r-1})$ to $N_{C|\PP^r}((r+1)p)$, which is also a direct sum of distinct line bundles.  Thus the case of even $k$ reduces to $k=4$ and the case of odd $k$ reduces to $k=3$.

    \subsubsection*{Step 3: $k=4$} To settle the case $k=4$, we induct on $r$. Specialize $C$ to a general genus $1$ curve $C'$ of degree $2r-3$ union a line $\ell$ meeting quasitransversely at $p$. Let $q$ be a point on $\ell$ and specialize $q$ to a general point on $C'$. Project $C'$ from $q$ and denote the resulting curve of degree $2r-4$ in $\PP^{r-1}$ by $\overline{C}$. By the pointing bundle exact sequence (see for example \cite[Equation (8)]{lv22}), we have 
$$0 \to \OO_{C'}(1)(2q + 3p) \to N_{C'|\PP^r} [2p \posmod q](p) \to N_{\overline{C}|\PP^{r-1}}(p+q) \to 0.$$ By induction, as long as $r \geq 5$, the latter bundle is a direct sum of $r-2$ distinct line bundles. The base case of the induction is the case $r=4$ and $d=6$ treated in Step 1 above. Hence $N_{C|\PP^r}$ cannot be a (possibly trivial)  extension of $\frac{r-1}{2}$ pairs of line bundles. We conclude that when $r-1$ divides $d$ and $r \geq 4$, then $N_{C|\PP^r}$ is a direct sum of distinct line bundles.

  \subsubsection*{Step 4: $k=3$ when the characteristic is not 2}  Now we consider the case $k=3$ by induction. 
    If $d=6$ and $r=5$ and the characteristic of the base field is not 2, then the degree $6$ line bundle $L$ embedding $C \subset \PP^5$ has 4 distinct square roots $L_i$ such that $L_i^{\otimes 2} \cong L$.  Each $L_i$ gives an embedding of $C$ to $\PP^2$ as a cubic curve. The Veronese embedding $S_i \subset \PP^5$ of $\PP^2$, then embeds the curve as our original $C$. Hence,  $C \subset \PP^5$ lies in $4$ Veronese surfaces $S_i$ corresponding to $L_i$. The normal bundle $N_{C|S_i} \cong L_i^{\otimes 3}$. Since $N_{C|\PP^r}$ is semistable and $L_i^{\otimes 3}$ are line subbundles of the same slope, we conclude that $N_{C|\PP^5} \cong \oplus_{i=1}^4 L_i^{\otimes 3}$ and the normal bundle is a direct sum of distinct line bundles. When the characteristic of the base field is 2, then $L$ has only 2 distinct square roots and this argument breaks down.

When the characteristic is different from 2,  specialize $C$ to  $C' \cup \ell$, where $C'$ is a general genus $1$ curve of degree $\frac{3(r-1)}{2}-1$ and $\ell$ is a line meeting $C'$ quasitransversely at a general point $p$ and which intersects a general 2-secant line $M$ joining the points $q_1$ and $q_2$ of $C'$. Projecting $C'$ from $M$, we obtain a curve $\overline{C} \subset \PP^{r-2}$, where $\overline{C}$ is a curve of degree $\frac{3(r-2)}{2}$. By induction on $r$, $N_{\overline{C}|\PP^{r-2}}$ is a direct sum of distinct line bundles. Now we use the pointing bundle exact sequence \cite[\S 3.2]{lv22}
    $$0 \to N_{C'\to M}[2p\posmod \ell](p) \to N_{C'|\PP^r}[2p \posmod \ell](p) \to N_{\overline{C}|\PP^{r-2}}(p+ q_1 + q_2) \to 0.$$
    By induction, $N_{\overline{C}|\PP^{r-2}}(p+q_1 + q_2)$ is a direct sum of $r-3$ line bundles of degree $\frac{3r+3}{2}$. By \cite[Lemma 6.3]{lv22}, $N_{C'\to M}[2p\posmod \ell](p)$ is semistable and has the same slope $\frac{3r+3}{2}$. Since $r \geq 7$ and $r-3$ of the line bundles are distinct, they cannot all occur in pairs. By considering the semicontinuity of the endomorphisms of the associated graded, we see that the number of line bundles that occur in pairs can only increase in the limit.  We conclude that for the general curve $C$, the bundle $N_{C|\PP^r}$ must be a direct sum of distinct line bundles. This concludes the proof of part (2) of the proposition.

   \subsubsection*{Step 5: The characteristic is 2, $r$ is odd, and $r-1$ does not divide $d$} In this case, we need to show that the normal bundle is a direct sum of $\frac{r-1}{2}$ indecomposable rank 2 bundles. Let $F$ denote the  Frobenius morphism on the elliptic curve $C$. We claim that $N_{C/\PP^r}(-1)$  is the pullback of a bundle under $F$. Using the Euler sequence we can express the dual as follows
   $$0 \to N_{C|\PP^r}^{\vee} (1) \to \OO_C^{r+1} \to \mathcal{P}^1(\OO_C(1))\to 0,$$ where $\mathcal{P}^1(\OO_C(1))$ is the bundle of principal parts of $\OO_C(1)$. Since the characteristic is 2, the bundle of principal parts can be expressed as $$\mathcal{P}^1(\OO_C(1))= F^* F_* \OO_C(1)$$ (see \cite[\S 3]{clv22}). Consequently, $N_{C|\PP^r}^{\vee} (1)$  and hence $N_{C/\PP^r}(-1)$ are pullbacks under $F$. Write $N_{C|\PP^r} (-1) = F^* K$ for a bundle $K$. Since $N_{C|\PP^r} (-1)$ has degree $k(r-1)$ and is semistable, $K$ must be semistable and by the Atiyah classification must have $\frac{r-1}{2}$ stable rank 2 bundles of odd degree $k$ in its Jordan-H\"{o}lder filtration.

    Suppose $L$ is a line bundle of odd degree $k$ on $C$. Then $F_* L$ is a rank 2 stable bundle of degree $k$ since for any line bundle $M$ of degree at least $\frac{k+1}{2}$ we have $$\Hom(M, F_* L) \cong \Hom(F^*M, L) =0.$$ Next observe that if $V$ is a stable bundle of rank 2 and odd degree $k$, then $F^* V$ is indecomposable. By the Atiyah classification and the observation that $F_*L$ is stable, we may write $V= F_*L$. Since the characteristic is 2, the bundle $F^* F_* L$ is the bundle of principal parts of $L$ and fits in the exact sequence
   $$0 \to L \cong L\otimes \Omega_C \to F^*V = F^* F_* L \to L \to 0.$$ To see that the extension is nontrivial, we observe that $$\Hom(F^*V, L) = \Hom(F^* F_* L, L) = \Hom(F_*L, F_*L) = \Hom(V,V).$$ The latter is 1-dimensional since $V$ is stable. We conclude that $F^*V$ is indecomposable.  

   Hence, $N_{C/\PP^r}(-1)$ and hence $N_{C/\PP^r}$ must be a direct sum of $\frac{r-1}{2}$ indecomposable rank 2 bundles. This concludes the proof of part (1) of the proposition.
\end{proof}

\begin{lemma}\label{lem-line}
Let $\mathcal{C} \to B$ be a family of nodal curves with nonsingular total space and generic fiber over a nonsingular curve $B$. Suppose that the central fiber $C_{b_0}$ is a union  $C' \cup \cup_{i=1}^{s} \ell_i$ of a smooth curve $C'$ together with $s$ disjoint 1-secant lines $\ell_i$ meeting $C'$ quasitransversely at points $p_i$. Then $N_{C'|\PP^r}[2p_1 \posmod \ell_1] \cdots [2p_s \posmod \ell_s](p_1 + \cdots + p_s)$ is a limit of the normal bundles $N_{\mathcal{C}_b|\PP^r}$.
\end{lemma}

\begin{proof}
Let $S'$ denote the total space of the family. Then the normal bundles of the fibers define a vector bundle $N$ on $S'$. The restriction of $N$ to $\ell_i$ is $\OO_{\ell_i}(2) \oplus \OO_{\ell_i}(1)^{\oplus r-2}$. The restriction of $N$ to $C'$ is $N_{C'|\PP^r} [p_1 \posmod \ell_1] \cdots [p_s \posmod \ell_s]$. Let $N_1$ be the  elementary modification of $N$ along $\ell_i$ with quotient $\OO_{\ell_i}(2)$ for every line $\ell_i$. By \cite[\S 2]{aly} (see also \cite[\S 2.2]{clv23}) if $V|D = F \oplus V'$, then $V[D \negmod F]|_D = \OO_D(D) \otimes F \oplus V'$. Hence, $N_1$ restricts to $\OO_{\ell_i}(1)^{\oplus (r-1)}$ along each $\ell_i$. Let $N_2 := N_1 (\sum_{i=1}^s \ell_i)$. Then $N_2$ restricts to each $\ell_i$ as the trivial bundle $\OO_{\ell_i}^{\oplus (r-1)}$. On the other hand, $N_2$ restricts to $C'$ as $N_{C'|\PP^r}[2p_1 \posmod \ell_1] \cdots [2p_s \posmod \ell_s](p_1 + \cdots + p_s)$.

   Since the $\ell_i$ are $(-1)$-curves in $S'$, we can blow them down to obtain a smooth surface $S$. Since $N_2|_{\ell_i}$ is trivial, the bundle $N_2$ is the pullback of a bundle from $S$.  This new family realizes $N_{C'|\PP^r}[2p_1 \posmod \ell_1] \cdots [2p_s \posmod \ell_s] (p_1 + \cdots + p_s)$ as the limit of $N_{\mathcal{C}_b|\PP^r}$.
\end{proof}

\begin{remark}
Although we will not need it in this paper, for completeness  we discuss normal bundles of general nondegenerate elliptic curves in $\PP^3$. 
\begin{enumerate}
    \item A degree $4$ genus 1 curve $C$ in $\PP^3$ is a complete intersection of two quadric hypersurfaces. Hence,  $N_{C|\PP^3} \cong \OO_C(2)^{\oplus 2}$.
    \item If the degree is odd and the characteristic of the base field is 2, then $N_{C|\PP^3}$ is the nontrivial extension of $\OO_C(2)$ by $\OO_C(2)$.
    \item Otherwise, $N_{C|\PP^3}$ is a direct sum of distinct line bundles.
\end{enumerate}
Assume that the degree $d$ of the elliptic curve $C$ is at least 5. 
Larson in \cite{la24} proves that $h^0(N_{C|\PP^3}(-2))=1$ if the characteristic is 2 and $d$ is odd, and  $h^0(N_{C|\PP^3}(-2))=0$ otherwise. By Step 5 of the proof of Proposition \ref{prop-g1distinct}, if the characteristic is 2 and $d$ is odd, $N_{C|\PP^3}$ is a nontrivial extension of a line bundle $L$ by itself. Since $h^0(N_{C|\PP^3}(-2))=1$, the line bundle $L$ must be $\OO_C(2)$. This proves (2).

Otherwise,  $N_{C|\PP^3}$ is either a direct sum of two distinct line bundles or it is a (possibly trivial) extension of a line bundle $L$ by itself. Since $\det (N_{C|\PP^3}) \cong \OO_C(4)$, in the latter case, $L$ differs from $\OO_C(2)$ by a 2-torsion line bundle. Since 2-torsion bundles on $C$ are discrete and by \cite[\S7 case $(5,1)$]{clv22}  and Step 2 of Proposition \ref{prop-g1distinct} $N_{C|\PP^3}$ can be specialized to an extension of $\OO_C(2)$ by itself,  $L$ must be $\OO_C(2)$. If $N_{C|\PP^3}$ were a possibly trivial extension of $\OO_C(2)$, then $h^0(N_{C|\PP^3}(-2))>0$ contradicting Larson's theorem. Hence, $N_{C|\PP^3}$ must be a direct sum of distinct line bundles. This proves (3).
\end{remark}

\section{Degenerations}\label{sec-degenerations}

In this section, we use three specializations to reduce the (semi)stability of the normal bundle of a general BN-curve to one of smaller degree or genus. 
We specialize a smooth BN-curve $C\subset \PP^r$ to the union of a BN-curve $C'$ and $r-1$ rational curves of degree $r-1$, all $(r+1)$-secant to $C$. We specialize $C$ to the union of a BN-curve $C'$ and a collection of 1 and 2-secant lines. Finally,  we specialize a genus 2 curve to the nodal union of two genus 1 curves.
\subsection{Attaching degree \texorpdfstring{$r-1$}{r-1} rational curves}
Let $C \subset \PP^r$ be a smooth BN-curve of degree $d$ and genus $g$. For $0\leq i\leq r-1$, let $R_i$ be a general rational curve of degree $r-1$ which is $(r+1)$-secant to $C$. 
Let $$C'= C \cup \bigcup_{i=1}^{r-1} R_i.$$

\begin{lemma}\label{ratlCurves}
 Let $\rho(d,g,r) \geq r-1$.
If $N_{C|\PP^r}$ is (semi)stable, then $N_{C'|\PP^r}$ is (semi)stable.
\end{lemma}
\begin{proof}
We show that the restriction $N_{C'\vert \PP^r}\vert_C$ is (semi)stable and that all the bundles $N_{C'\vert \PP^r}\vert_{R_i}$ are semistable, so the result follows from Lemma \ref{lem:naive}.

To show $N_{C'\vert \PP^r}\vert_C$ is (semi)stable, we specialize $C'$. Let $H$ be a general hyperplane section of $C$ and let $\Gamma=\{p_1,\ldots,p_{r+1}\}$ be $r+1$ distinct points in $C\cap H$. Specialize each $R_i$ to a general rational normal curve $R_i^0$ through the points of $\Gamma$, and let $C'^0=C\cup \bigcup_i R_i^0$. 

We first claim that at any point $p\in \Gamma$, the tangent spaces $T_pR_i^0$ are linearly independent in $H$. To show this, we note that given a general tangent direction $v$ at a point $p_i\in \Gamma$, there is a rational normal curve $R$ through all the points of $\Gamma$ that also has tangent direction $v$ at $p_i$. Hence, given general curves $R_1,\ldots, R_a$ with $a < r-1$, it is possible to find $R_{a+1}$ whose tangent line does not lie in the linear span of the tangent lines of $R_1,\ldots,R_a$ at any $p_i$. By openness, we then have that the tangent directions of all the $T_{p_i}R_j$ are linearly independent. This establishes the claim.

We observe that the limit of $N_{C'\vert \PP^r}\vert_C$ as $C'$ approaches $C'^0$
 is 
 \[N_{C\vert \PP^r}(p_1\posmod R_1)\cdots (p_1\posmod R_{r-1})\cdots(p_{r+1}\posmod R_1)\cdots (p_{r+1}\posmod R_{r-1}),
\]
which is isomorphic to the (semi)stable bundle $N_{C\vert \PP^r}(p_1+\cdots+p_{r+1})$.

Meanwhile, since $\rho(d,g,r) \geq r-1 > 1$, the curve $C$ is not an elliptic normal curve. Hence, the bundles $N_{C'\vert \PP^r}\vert_{R_i}$ are all semistable by \cite[Lemma 5.8]{lv22}.
By Lemma \ref{lem:naive}, $N_{C'\vert \PP^r}$ is (semi)stable.
\end{proof}

\begin{corollary}\label{cor-rnc}
    If $(d,g,r)$ is a BN (semi)stable triple and $\rho(d,g,r) 
    \geq r-1$, then $(d+ (r-1)^2, g + r(r-1), r)$ is a BN (semi)stable triple.
\end{corollary}

\begin{proof}
   By Lemma \ref{reducibleBNCurves}, the union of a general BN-curve $C$ of degree $d$ and genus $g$ in $\PP^r$ with $r-1$ general $(r+1)$-secant rational curves of degree $r-1$ is a BN-curve of degree $d+ (r-1)^2$ and genus $g + r(r-1)$. The corollary follows by Lemma \ref{ratlCurves}.
\end{proof}

\subsection{Attaching lines to curves}
Next, following \cite{clv22, lv22}, we specialize our BN-curve of degree $d$ and genus $g$ to a union of  a BN-curve with smaller degree or genus and several lines. The following lemma allows us to show that the normal bundle of the resulting configuration is (semi)stable.

\begin{lemma}\label{lines}
Let $C\subset \PP^r$ be a smooth curve and $\{\ell_i\}_{1\leq i\leq e}$ be disjoint lines such that $$C'=C\cup \bigcup_{i=1}^e \ell_i$$ is a connected nodal curve with nodes $p_1,\ldots,p_M$. For each $j$, let $\ell(p_j)$ be the line containing $p_j$. If 
\[
N_{C\vert \PP^n}[2p_1\posmod \ell(p_1)]\cdots [2p_M\posmod \ell(p_M)]
\]
is (semi)stable, then $N_{C'\vert \PP^r}$ is (semi)stable.
\end{lemma}

\begin{proof}
Let $\nu:C\sqcup \bigsqcup_{i=1}^e \ell_i\ra C'$ be the normalization and let $S$ be a subbundle of $\nu^* N_{C'\vert \PP^r}$ Consider $S\vert_C\subset N_{C'\vert \PP^n}\vert_C \cong N_{C\vert \PP^r}[p_1 \posmod \ell(p_1)]\dots[p_N\posmod \ell(p_N)]$, and set $N'= N_{C\vert \PP^r}[2p_1\negmod \ell(p_1)]\dots[2p_N\negmod \ell(p_N)](p_1+\ldots+p_N)$, so $N'$ is the subbundle of $N_{C'\vert \PP^r}\vert_{C}$ that does not smooth the nodes. We have an exact sequence
\[
0\ra N' \ra N_{C'\vert \PP^r}\vert_C \ra \bigoplus_{1\leq i \leq N} T_{p_i}\PP^r/T_{p_i}C'\ra 0.
\]
Let $S'$ be the preimage of $S\vert_C$ in $N'$. Assume $S\vert_C$ has degree $a$ and rank $s$, so $S'$ has degree $a-sN+\sum_{1\leq j \leq N}\delta_i$, where  $\delta_j$ is given by
\[
\delta_j=\twopartdef{1}{S(-p_j)\vert_{p_j} \text{ includes the normal direction along } \ell(p_j)}{0}{\text{otherwise}}.
\]
We now consider the restrictions $S\vert_{\ell_i}$. We have 
$N_{C'\vert \PP^r}\vert_{\ell_i}\cong N_{\ell_i\vert \PP^r}([p_j\posmod C])_{p_j\in \ell_i}$, so we have an exact sequence
\[
0\ra N_{\ell_i\vert \PP^r}\ra N_{C'\vert \PP^r}\vert_{\ell_i}\xrightarrow{\rho} \bigoplus_{p_j\in \ell_i} \OO_{p_j}\ra 0
\]
Given $S\vert_{\ell_i}\subset N_{C'\vert \PP^r}\vert_{\ell_i}$, the preimage of $S\vert_{\ell_i}$ in $N_{\ell_i\vert \PP^r}\cong \OO(1)^{\oplus n-1}$ has degree at most $s$. For $1\leq j \leq N$, define $\epsilon_{j}$ by
\[
\epsilon_{j}=\twopartdef{1}{\rho:S\vert_{\ell(p_j)} \ra \OO_{p_j}\text{ nonzero} }{0}{\text{otherwise}}
\]
We have
\[
\deg(S\vert_{\ell_i}) \leq r+ \sum_{p_j\in \ell_i} \epsilon_j
\]
Likewise, for $1\leq j\leq N$, we set 
\[
\gamma_j = \twopartdef{0}{S\vert_{\ell(p_j)}\vert_{p_j}=S\vert_{C}\vert_{p_j}}{1}{\text{otherwise}}.
\]
We observe that if $S\vert_{\ell(p_j)}\vert_{p_j}=S\vert_C\vert_{p_j}$, then $\delta_{j}=\epsilon_{j}$. Hence,  for each $1\leq j\leq N$, we have $\delta_j+\gamma_j \geq \epsilon_j$. Set $\delta= \sum_{i=1}^N \delta_i$, $\epsilon = \sum_{j=1}^N \epsilon_j$, and $\gamma = \sum_{j=1}^N \gamma_j$, so in combination, we have $\delta + \gamma \geq \epsilon$.

 Since $S$ has degree at most
$
a+se+\epsilon,
$
the degree of $S$ is at most
\[
a+se+\delta+\gamma,
\]
and we have
\[
\mu^{adj}_{C'}(S)\leq \frac{1}{s}(a+se+\delta).
\]
We now bound $a+\delta$ from above.
Since $N'$ is (semi)stable of slope $\frac {\deg(N_{C\vert \PP^r})-N(r-3)}{r-1}$, we have 
\[
\frac{a-Ns+\delta}{s} \leqor \frac {\deg(N_{C\vert \PP^r})-N(r-3)}{r-1},
\]
so
\[
\frac{a+se+\delta}{s} \leqor \frac {\deg(N_{C\vert \PP^r})+2N+e(r-1)}{r-1}.
\]
The right-hand side of the above is exactly $\mu(N_{C'\vert \PP^r})$, so we have
\[
\mu^{adj}_{C'}(S) \leqor \mu(N_{C'\vert\PP^r}).
\]

\end{proof}

Applying this lemma, we can attach a network of lines to a BN-curve and retain (semi)stability of normal bundles. More precisely, we have the following.
\begin{proposition}\label{addingLines}
If $(d,g,r)$ is BN (semi)stable and $a,b,c$ are positive integers satisfying $a+b=c(r-1)$ and $b \leq \min(a,\binom{c}{2})$ then $(d+a,g+b,r)$ is BN (semi)stable.
\end{proposition}
\begin{proof}
Let $C'$ be a union of a general BN-curve $C$ of triple $(d,g,r)$ and $a$ lines $\ell_1,\ldots,\ell_{a}$, $b$ of which $\ell_1,\ldots, \ell_b$ are 2--secant lines to $C$ through the general points $p_1,q_1, \ldots, p_b,q_b$, respectively, and $\ell_{b+1},\ldots,\ell_a$ are general 1-secant lines to $C$ through the general points $p_{b+1},\ldots, p_a$, respectively. By Lemma \ref{lines}, $N_{C'\vert \PP^r}$ is (semi)stable if the  bundle \[N\cong N_{C\vert \PP^r}[2p_1\posmod \ell_1] \cdots[2p_a\posmod \ell_a] [2q_1\posmod \ell_1]\cdots[2q_b\posmod \ell_b]\] is (semi)stable. Pick $c$ general points $\Gamma :=\{r_1,\ldots,r_c\}$ on $C$. We first degenerate the 2-secant lines to pass through distinct pairs of $\Gamma$.  Then we degenerate the 1-secant lines so that  $r-1$ lines are incident at each point of $\Gamma$ and at each point $r_i$ the lines and the tangent line to $C$ at $r_i$ span $\PP^r$. Then the limit of $N$ is  the bundle $N_{C\vert \PP^r}(2r_1+\cdots+2r_c)$. Since $N_{C\vert \PP^r}$ is (semi)stable, this degeneration of $N$ is also (semi)stable. By Lemma \ref{lines}, we conclude that $N_{C'|\PP^r}$ is (semi)stable. 

By Lemma \ref{reducibleBNCurves}, $C'$ is a BN-curve and corresponds to the triple $(d+a,g+b,r)$. We conclude that if $(d,g,r)$ is BN (semi)stable, then  $(d+a,g+b,r)$ is BN (semi)stable.
\end{proof}

\subsection{Degenerating genus 2 curves}
A semistable bundle whose degree and rank are coprime is stable. When $r$ is odd,  $N_{C\vert \PP^r}$ has even degree and rank, so stability requires more work. In this section, we prove Proposition \ref{genus2Stab}, which establishes the stability of normal bundles to general genus 2 curves of sufficiently high degree,  by specializing a genus 2 curve $C$ to a general union of two 1-secant genus 1 curves $C_1 \cup C_2$ whose degrees are chosen so $N_{C_1\cup C_2\vert \PP^r}\vert_{C_1}$ is stable and $N_{C_1\cup C_2\vert \PP^r}\vert_{C_2}$ is semistable.

\begin{proposition}\label{prop-g1mod}
Let $C$ be a general nondegenerate genus 1 curve of degree $d$ in $\PP^r$, $r \geq 4$. Let $p_1, \dots, p_t \in C$ be general points and $q_1, \dots, q_t \in \PP^r$ be general points. Then $N_{C\vert \PP^r}[p_1\posmod q_1]\cdots [p_t \posmod q_t]$ is semistable.
\end{proposition}
\begin{proof}
    We will prove the proposition by induction on the number of elementary modifications. The base case $t=0$ is Theorem \ref{genus1}. Suppose the proposition holds for $t<t_0$. Let $N_1 := N_{C|\PP^r}[p_1\posmod q_1]\cdots [p_{t_0-1} \posmod q_{t_0-1}]$. For ease of notation, set $p= p_{t_0}$ and $q=q_{t_0}$. Let $N_2 := N_1 [p \posmod q]$. By induction, $N_1$ is semistable. We want to show that $N_2$ is semistable. Suppose that $N_2$ is destabilized by a subbundle $S'$ of rank $s$ and degree $e+1$. Then $S'$ is an extension of a subbundle $S$ of $N_1$ of rank $s$ and degree $e$ by $\OO_p$. By comparing slopes, we have
    $$\mu(N_1) = \frac{d(r+1) + t_0 -1}{r-1} \geq \mu(S) = \frac{e}{s}$$
    $$\mu(S') = \frac{e+1}{s} > \mu(N_2) = \frac{d(r+1) + t_0}{r-1}.$$ Combining and multiplying by $s(r-1)$, we have
    $$ds(r+1)+st_0 -s \geq e(r-1) > ds(r+1)+st_0 -r+1.$$ Since the difference between the left and right hand sides of the inequality is less than $r-1$, there can be at most one integer $e$ satisfying this inequality. Set $\delta = 2ds+st_0 -s \mod{r-1}$. Then $e= ds + \frac{2ds+st_0 - s- \delta}{r-1}$ and $e$ can only take this value if $\frac{2ds+st_0 - s- \delta}{r-1} + 1 > \frac{2ds+st_0}{r-1}$, equivalently, if $\delta< r-1-s$.

    We will show that as we vary $S\subset N_1$ over all bundles with the given discrete invariants $(s,e)$,  $S\vert_p$ cannot contain a general element of $N_1$. Consequently, if $q$ is general, the inclusion $S\hookrightarrow N_1$ must fail to drop rank at $p$, and hence cannot induce the desired inclusion $S'\hookrightarrow N_2$. 
    
    First suppose that $\delta\geq 1$. Set $z=\gcd(e,s)$, and suppose $S\cong S_1\oplus \cdots \oplus S_c$, where each $S_i$ is an indecomposable extension of $z_i$ copies of the same bundle $E_i$. The numbers $c$ and $z_i$ are fixed as $q$ varies in an open set, and the moduli space of indecomposable bundles of rank $\frac{z_i s}{z}$ and degree $\frac{z_i e}{z}$ has dimension 1. Since $N_1$ is semistable with $\mu(N_1) > \mu(S_i)$, we have $\Ext^1(S_i, N_1)=0$. Consequently, $\hom(S_i, N_1)=\chi(S_i, N_1)=\frac{\delta z_i}{z}.$ The set of all possible images of $S_i$ in the vector space $N_1\vert_p$ has dimension at most $\frac{(\delta+s) z_i}{z}-1$, and hence the image of all indecomposable semistable bundles of the same rank and degree as $S_i$ has dimension at most $\frac{(\delta+s) z_i}{z}$. Combining, the set of all images of a vector bundle that decomposes like $S$ in $N_1$ has dimension at most $\delta+s$, which is  bounded above by $r-2$. Hence, if $v\in N_1\vert p$ is general, no $S$ decomposing in the same way as the $S$ here can contain $v$.

    We now suppose $\delta=0$. In this case, $S\subset N_1$ has slope $\mu(N_1)$. We decompose $N_1\cong M_1\oplus \cdots \oplus M_c$, where each $M_i$ is an iterated extension of a stable bundle $W_i$. This decomposition is unique up to reordering since $\Hom(M_i, M_j)=0$ for $i\neq j$. Any subbundle $S\subset N_1$ of slope $\mu(N_1)$  splits as $S_1\oplus \cdots \oplus S_c$, where $S_i\subset M_i$ is a subbundle of slope $\mu(N_1)$.  
    
    If $t_0 -1\geq 1$, we claim that we may assume that every $M_i$ is isomorphic to $W_i$, or equivalently, $N_1$ is polystable with no repeated stable factors. Let $N_0$ be the bundle with the first $t_0 -2$ modifications. Since $\mu(N_0) < \mu(N_1)$, $N_0$ admits a nonzero homomorphism to every polystable bundle with the same slope as $N_1$. Hence, $N_0$ admits a nonzero homomorphism to every polystable bundle $N_1'$ with $\mu(N_1') = \mu(N_1)$ which is the direct sum of distinct stable bundles. Since the homomorphism $N_0 \to N_1$ is injective, the general homomorphism to $N_1'$ is also injective. Hence, $N_1'$ appears as an elementary modification of $N_0$. Since $N_1$ is a general elementary modification of $N_0$, we conclude that we may assume that $N_1$ is a direct sum of distinct stable bundles.

    If $t_0=1$,
by the proof of Theorem \ref{genus1}, either every $M_i$ is isomorphic to $W_i$, or every $M_i$
 is an extension---possibly trivial---of $W_i$ by $W_i$. Since $S$ is a proper subbundle of $N_1$, at least one of the $S_i$ is a proper subbundle of $M_i$. Assume   $S_1\subset M_1$ is a proper subbundle. 

If $M_i\cong W_i$ for all $i$, then $S_1=0$, and even as $S$ varies, its restriction to $p$ cannot contain a general $v\in N_1\vert_p$. 

If each $M_i$ is the nontrivial extension of $W_i$ by $W_i$, then $\hom(W_1, M_1)=1$. The only possible choices for a proper $S_1\subset M_1$ are 0 or $W_1$.  In the first case, the restriction of $S$ to $p$ cannot contain a general $v\in N_1\vert_p$. In the latter case, $S_1\subset M_1$ is canonical, and if $v_1\in M_1\setminus W_1$, $S$ cannot contain any $v\in N_1\vert_p$ whose projection to $M_1$ is $v_1$.

If  each $M_i\cong W_i \oplus W_i$, then by Proposition \ref{prop-g1distinct} we have $\rk(W_i)\geq 2$. Hence, we have either $S_1\cong W_1$ or $S_1=0$. In the latter case, we once again have that $S\vert_p$ cannot contain any element of $N_1\vert_p$ that has a nonzero $M_1$ component. In the former case, let $w_1,w_2$ be linearly independent vectors in $W_1\vert p$; then $S$ cannot contain any vector whose restriction to $M_1\vert_p\cong (W_1\oplus W_1)\vert_p$ is $(w_1,w_2)$. In either case, we cannot extend $\OO_p$ by a bundle in the same family as $S$ to get a destabilizing subbundle of $N_2$ for $q$ general.
\end{proof}

\begin{proposition}\label{genus2Stab}
Let $d$ be an integer such that there exist integers $d_1,d_2\geq r+1$ such that $d=d_1+d_2$, $\gcd(r-1, 2d_1+1)=1$. Then $(d,2,r)$ is BN stable.
\end{proposition}

\begin{proof}
We degenerate $C$ to two general genus 1 curves $C_1$, $C_2$ of degree $d_1$ and $d_2$, respectively, intersecting in a single point. We choose the degrees $d_1$ and $d_2$ such that the hypotheses of this proposition hold for them. Then $N_{C_1\cup C_2\vert \PP^r}\vert_{C_1}$ and $N_{C_1\cup C_2\vert \PP^r}\vert_{C_2}$ are both semistable by Proposition \ref{prop-g1mod}. In addition, $N_{C_1\cup C_2\vert \PP^r}\vert_{C_1}$ has degree $d_1(r+1)+1$, which is coprime to its rank $r-1$ by hypothesis. Hence, $N_{C_1\cup C_2\vert \PP^r}\vert_{C_1}$  is stable. By Lemma \ref{lem:naive},  $N_{C_1\cup C_2\vert \PP^r}$ is  stable. Finally, $C_1\cup C_2$ is a BN-curve by Proposition \ref{230}, so $(d,2,r)$ is BN stable. 
\end{proof}
\begin{definition}\label{definition-b2r}
Given a positive integer $r$, let $b_2(r)$ be the least integer such that for all $d\geq b_2(r)$, there exist positive integers $d_1,d_2$ such that 
$d= d_1 + d_2$, $d_1, d_2 \geq r+1$ and $\gcd(r-1,2d_1+1)=1$ 
\end{definition}
Proposition \ref{genus2Stab} gives that $(d,2,r)$ is BN stable for $d\geq b_2(r)$. The integer $b_2(r)$ is typically close to $2r+2$, and we can precisely bound it in all cases.

\begin{proposition}\label{b2Bounds} Let $r \geq 4$. 
\begin{enumerate}
\item If $5 \nmid r-1$, then $b_2(r)=2r+2$.

\item Let $p \geq 5$ be the smallest prime which does not divide $r-1$. Then $b_2(r) \leq 2r+\frac{p-1}{2}$.

\item For any $r\geq 8$ even, $b_2(r)\leq \frac{5r-6}{2}$. For any $r\geq 8$ odd, $b_2(r)\leq \frac{5r-3}{2}$.

\item $b_2(r)=2r+O(\log(r))$
\end{enumerate}
\end{proposition}
\begin{proof}
\textbf{(1)}: Assume $5 \nmid r-1$, then  Proposition \ref{genus2Stab} applies with $d_1=r+1$ since $r-1$ and $2r+3$ are relatively prime. Hence, $b_2(r) = 2r+2$.

\textbf{(2)}: If $p \nmid r-1$ and $p \geq 5$, then Proposition \ref{genus2Stab} applies with $d_1=r-1 + \frac{p-1}{2}$ since $r-1$ and $2r-2 +p$ are relatively prime. Hence, $b_2(r) \leq 2r+\frac{p-1}{2}$.

\textbf{(3)}: Suppose $r\geq 12$ is even.
Let $d'_{1}$ be the least integer $\geq 2$ such that $2d'_1+1$ is coprime to $r-1$. Since $r-1\geq 11$ is odd, we have $\phi(r-1)\geq 8$, so there are at least 4 integers $i$ with $2\leq i\leq r-4$ with $\gcd(2i+1,r-1)=1$. Since replacing $i$ with $r-2-i$ preserves coprimality with $r-1$, there are at least 2 such integers in the range $[2,\frac{r-2}{2})$, and $d'_1\leq \frac{r}{2}-3$. Hence,  given $d\geq \frac{5r-6}{2}$, we take $d_1=d'_1+r-1$.

If $r\geq 9$, $r\neq 13$ is odd, we have $\phi(\frac{r-1}{2})\geq 4$. Let $d_{1,1}<d_{1,2}<\frac{r-1}{2}$ be the least integers at least 2 such that $\gcd(2d_{1,i}+1,\frac{r-1}{2})=1$, and take $d_1=d_{1,1}+r-1$ unless $d\equiv d_{1,1}\pmod{r-1}$, in which case take $d_1=d_{1,2}+r-1$. This works for any $d\geq d_{1,2}+2r$, hence for any $d\geq \frac{5r-3}{2}$. 

\textbf{(4)}: Let $p_1<p_2$  be the smallest positive integers  such that $\gcd(2p_1+3,r-1)=\gcd(2p_2+3,r-1)=1$. Applying the previous lemma with $d_1=p_1+r$, unless $d\equiv p_1+1 \pmod r-1$, in which case we take $d_1=p_2+r$, we have that $b_2(r)\leq p_2+2r+1$

 The integers $2p_1+3$ and $2p_2+3$ are bounded above by the two smallest primes at least 5 that fail to divide $r-1$. By the Prime Number Theorem, both such primes are $O(\log(r))$, so $b_2(r)-2r=O(\log(r))$ as well.
\end{proof}

\begin{remark}\label{rem-be}
    Using more precise number theoretic estimates one can improve Proposition \ref{b2Bounds} (3) significantly. For example, if $x \geq 3275$, then there exists a prime number $p$ such that $x \leq p < 1.01 x$ \cite[\S 4]{dusart}. Hence, if $r \geq 1636$, there exists a prime number $p$ such that $2r+3 \leq p < 2.02 r + 3.03$. We can then take $d_1 = \frac{p-1}{2}$ and see that $b_2(r) \leq 2.01 r + 2.015$.
\end{remark}

\subsection{Configurations of elliptic curves} In this section, we will use configurations of elliptic curves to extend semistability to higher genus curves (see \cite{br99} for a similar construction).

For $e \geq r+1$, set  $$u_e = \left \lfloor \frac{(r+1)e}{2(r-1)} \right \rfloor \quad \mbox{and} \quad u= u_{r+1}.$$ 
Let $E_0$ be an elliptic curve of degree $e_0 \geq r+1$. For $1 \leq i \leq I$, let $r+1 \leq e_i$ be integers and let  $0< j_i \leq \min(u_{e_i}, u_{e_0})$ be positive integers.  Let $p_{i, 1}, \dots, p_{i, j_i}$ be general points on $E_0$ for $1 \leq i \leq I$ and let $E_i$ be a general elliptic curve of degree $e_i$ meeting $E_0$ quasitransversely at $p_{i,1}, \dots, p_{i, j_i}$. Let $C$ be the union of $E_0$ with $E_i$ for $1 \leq i \leq I$. By Proposition \ref{230} and Corollary \ref{cor-g1bn}, C is a BN-curve of degree $ \sum_{i=0}^I e_i$ and genus $1 +\sum_{i=1}^I j_i$.

\begin{proposition}\label{prop-g1config}
    The normal bundle of $C$ is semistable. Furthermore, if $r-1$ is relatively prime to one of   $2e_0 + \sum_{i=1}^I j_1$, or $2e_i+j_i$  for some $1 \leq i \leq I$, then the normal bundle of $C$ is stable.
\end{proposition}
\begin{proof}
By Theorem \ref{genus1}, $N_{E|\PP^r}$ of a general degree $d$ genus 1 BN-curve is semistable. Consequently, $H^1(E, N_{E|\PP^r}(-2p_1 \cdots -2 p_{u_d}))=0$. We conclude that there exists genus 1 curves of degree $d$ passing through $p_1, \dots, p_{u_d}$ general points in $\PP^r$ with general tangent lines at these points.  

For $1 \leq i \leq I$, let $C_i= E_i \cup E_0$. Since the tangent lines at the attaching points are general lines, the bundle $N_{C_i|\PP^r}|_{E_i}$ is a general modification of $N_{E_i|\PP^r}$ and, by Proposition \ref{prop-g1mod}, is semistable. Since $N_{C|\PP^r}|_{E_i} = N_{C_i|\PP^r}|_{E_i}$, we conclude that it is also semistable. Finally, $N_{C|\PP^r}|_{E_0}$ is a general modification of $N_{E_0|\PP^r}$ and is semistable by Proposition \ref{prop-g1mod}. We conclude that $N_{C|\PP^r}$ is semistable.  If $r-1$ and $2e_i + j_i$ are relatively prime, then $N_{C|\PP^r}|_{E_i}$ is stable, so $N_{C|\PP^r}$ is stable. Similarly, if $r-1$ and $2e_0+ \sum_{i=1}^I j_i$ are relatively prime, then  $N_{C|\PP^r}|_{E_0}$ is stable, so $N_{C|\PP^r}$ is stable.
\end{proof}

\begin{corollary}\label{cor-1configgen}
    Let $a_1 \geq a_2 \geq \cdots \geq a_I >0$ be positive integers. Let $$e_i = \left \lceil \frac{2a_i (r-1)}{r+1} \right\rceil.$$ Set $g = 1 + \sum_{i=1}^I a_i$ and $e = e_1 + \sum_{i=1}^I e_i$. Then $(d,g,r)$ is semistable for all $d \geq e$. Furthermore, $(d,g,r)$ is stable if $r-1$ is relatively prime with one of $2e_i + a_i$, for $1 \leq i \leq I$.
\end{corollary}

\begin{proof}
Specialize a general BN-curve of degree $d$ and genus $g$ to a union of elliptic curves $E_0$ of degree $d - \sum_{i=1}^I e_i$ and elliptic curves $E_i$ of degree $e_i$ meeting $E_i$ at $a_i$ general points with general tangent lines at the points of intersection. Since $d - \sum_{i=1}^I e_i\geq e_1$ by assumption, Propositions \ref{prop-g1config} and \ref{prop:stab-open} imply the corollary. 
\end{proof}

\begin{corollary}\label{cor-g2}
Let $r \geq 4$. Then a general  BN-curve of degree $d \geq 2r$ and genus 2 has semistable normal bundle.
\end{corollary}

\begin{proof}
Let $C$ the union of a general elliptic curve $E$ of degree $d-r+1$ with a general 2-secant rational normal curve $R$ of degree $r-1$ contained in a hyperplane. By Lemma \ref{reducibleBNCurves}, this curve is a BN-curve. Then $N_{C|\PP^r}|_R$ is balanced, hence semistable. Since the two tangent lines to the rational normal curve can be chosen arbitrarily, by Proposition \ref{prop-g1mod}, $N_{C|\PP^r}|E$ is semistable. The corollary follows from Proposition \ref{prop:stab-open}.
\end{proof}

We can simplify the numerics at the expense of weakening the bounds.  When $$(k-1)u < g-1 \leq ku,$$ we can take a configuration of $(k-1)$ general elliptic normal curves  $E_i$ that are $u$-secant to  an elliptic curve $E_0$ of degree $d- k(r+1)$ and an additional elliptic normal curve which is  $g-1-(k-1)u$ secant to $E_0$. Proposition \ref{prop-g1config} yields the following corollary.

\begin{corollary}\label{cor-sse}
    Let $(k-1)u < g-1 \leq ku$. Then the normal bundle of a general BN-curve of genus $g$ and degree $d \geq (k+1)(r+1)$ is semistable.
\end{corollary}

A minor modification by adding a general elliptic curve of degree $b_2(r) - r- 1$ which is 1-secant to $E_0$ and making one of the $E_i$ one less secant to make the union have genus $g$, we obtain the following corollary.

\begin{corollary}\label{cor-se}
    Let $(k-1)u < g-1 \leq ku$. Then the normal bundle of a general BN-curve of genus $g$ and degree $d \geq b_2(r) + k(r+1)$ is stable.
\end{corollary}

\section{Proof of Theorems \ref{largeMain} and \ref{thm-introstab}}\label{largeSection}
 In this section, we deduce Theorems \ref{largeMain} and \ref{thm-introstab} from the results in Sections \S \ref{sec-prelim}, \S \ref{sec-elliptic} and \S \ref{sec-degenerations} purely numerically. For the reader's convenience, we collect our  results in the following remark.

\begin{remark}\label{rem-num}
\begin{enumerate}
     \item If $\rho(d,g,r)\geq 0$, $r-1$ divides $2d+2g-2$, and $$(d,g,r)\not\in \{(5,2,3), (6,4,3), (7,2,5) \},$$ then $(d,g,r)$ is BN semistable $($Corollary \ref{interpolation}$)$.
    \item  If $d-1\geq r\geq 3$, then $(d,1,r)$ is BN semistable $($Theorem \ref{genus1}$)$. 
    \item If $(d,g,r)$ is BN (semi)stable and $\rho(d+(r-1)^2,g+r(r-1), r)\geq 0$, then $(d+(r-1)^2,g+r(r-1),r)$ is BN (semi)stable $($Corollary \ref{cor-rnc} obtained by specializing the curve to a union of a BN-curve $C$ of type $(d,g,r)$ and $(r-1)$ rational normal curves of degree $r-1$ which are $(r+1)$-secant to $C)$.
    \item If $(d,g,r)$ is BN (semi)stable, $a,b,c\in \ZZ_{\geq 0}$ satisfy $a+b=c(r-1)$, and $b\leq \min(a,\binom{c}{2})$, then $(d+a,g+b,r)$ is BN (semi)stable $($Proposition \ref{addingLines} obtained by specializing the curve to a BN-curve $C$ of type $(d,g,r)$ union $b$ general 2-secant lines to $C$ and $(b-a)$ general 1-secant lines to $C)$.
     \item If $r\geq 4$, $(d,2,r)$ is BN semistable if $d \geq 2r$ and BN stable if $d\geq b_2(r)$ $($Proposition \ref{genus2Stab} and Corollary \ref{cor-g2} obtained by specializing the genus 2 curve to a union of two genus 1 curves$)$. 
     \item Express $g = 1+ \sum_{i=1}^I a_i$ for positive integers $a_1 \geq \cdots \geq a_I >0$. Let $e_i$ be integers  such that $e_i \geq  \frac{2(r-1) a_i}{r+1}$. Then $(d,g,r)$ is BN semistable if $d \geq e_1 + \sum_{i=1}^I e_i$. Furthermore, if $r-1$ and $2e_i + a_i$ is relatively prime for some $i$, then $(d,g,r)$ is stable (Corollary \ref{cor-1configgen} obtained by specializing the curve to a union of $I+1$ genus 1 curves $E_i$ for $0 \leq i \leq I$, where $E_i$ is $a_i$-secant to $E_0$ for $1 \leq i \leq I$ at general points with general tangent directions). 
    \item The triple $(7,2,4)$ is BN stable if the characteristic of the base field is not 2 (Theorem \ref{724}).
\end{enumerate}
\end{remark}

\begin{lemma}\label{243}
Suppose $(d', g_0,r)$ is BN (semi)stable for all $d'\geq d_0$. Let $q, b$ be nonnegative integers and let $a$ be the minimal positive integer such that there exists an integer $c$ with $a+b = c(r-1)$ and $b \leq \min(a,\binom{c}{2})$.  Let  $g=g_0+qr(r-1)+b$.  Then if $(d,g,r)$ is a BN triple with $d\geq d_0+a+q(r-1)^2$, then $(d,g,r)$ is BN (semi)stable.
\end{lemma}
\begin{proof}
Apply Remark \ref{rem-num} (4) to  the BN (semi)stable triple $(d-q(r-1)^2-a, g_0, r)$ to obtain that $(d-q(r-1)^2,g-qr(r-1),r)$ is BN (semi)stable. Then apply Remark \ref{rem-num} (3) $q$ times to deduce that $(d,g,r)$ is BN (semi)stable.
\end{proof}
\begin{lemma}\label{249}
Suppose $(d, g_0,r)$ is BN (semi)stable for all $d\geq d_0$. Set
\[
q_0=\left \lceil \frac{2\rho(d_0,g_0,r)+5r^2-7r-2}{2(r-1)}\right \rceil.
\]
Then if $g\geq  g_0+\binom{r-1}{2}+q_0r(r-1)+1$ and $(d,g,r)$ is a BN triple, then $(d,g,r)$ is BN (semi)stable.
\end{lemma}
\begin{proof}
Set $\rho_0=\rho(d_0,g_0,r)$, and write $g=g_0+qr(r-1)+b$ for integers $q$ and $\frac{(r-1)(r-2)}{2}+1\leq b \leq \frac{(r-1)(r-2)}{2}+r(r-1)$.
Let $a$ be the minimal integer with $a\geq b$ and $a+b$ divisible by $r-1$.

We have
\[
\rho(d-q(r-1)^2-a, g_0,r)=\rho(d,g,r)+q(r-1)-b-(r+1)(a-b).
\]
We have $a-b \leq r-2$ and  $b \leq \binom{r-1}{2}+r(r-1)$.  The bound on $g$ in the hypothesis is  $q\geq q_0$. Hence, we obtain the inequality
\[
\rho(d-q(r-1)^2-a, g_0,r) \geq \rho(d,g,r)+q_0(r-1)-\binom{r-1}{2}-r(r-1)-(r+1)(r-2)
\]
which simplifies to
\[
\rho(d-q(r-1)^2-a, g_0,r) = \rho(d,g,r)+q_0(r-1)-\frac{5r^2-7r-2}{2}.
\]
By the bound on $q_0$ in the hypothesis, we then have $\rho(d-q(r-1)^2-a, g_0,r)\geq \rho(d_0,g_0,r)$, and $d-q(r-1)^2-a\geq d_0$, so Lemma \ref{243} implies $(d,g,r)$ is BN (semi)stable.
\end{proof}
\begin{theorem}
Let $(d,g,r)$ be a BN triple. If 
\begin{equation}\label{277}
g\geq \binom{r-1}{2}+2+\left \lceil \frac{5r^2-7r}{2(r-1)}\right\rceil r(r-1)
\end{equation}
then $(d,g,r)$ is BN semistable. If
\begin{equation}\label{279}
g \geq \binom{r-1}{2}+3+\left \lceil\frac{2(r+1)b_2(r)+3r^2-13r-2}{2(r-1)}\right \rceil r(r-1),
\end{equation}
with $b_2(r)$ the integer defined after Proposition \ref{genus2Stab}, then $(d,g,r)$ is BN stable.
\end{theorem}
\begin{remark}
The bound (\ref{277}) is on the order of $g\geq \frac{5r^3}{2}$, while the bound (\ref{279}) is on the order of $g\geq \frac{7r^3}{2}$.
\end{remark}
\begin{proof}
We apply Lemma \ref{249} to  the genus 1 and 2 cases in Remark \ref{rem-num} (2) and (5). For semistability, we have that $(d,1,r)$ is BN semistable for $d\geq r+1$. Since $\rho(r+1,1,r)=1$, the formula in Lemma \ref{249} gives 
\[
q_0=\left \lceil \frac{5r^2-7r}{2(r-1)}\right \rceil. 
\]
For stability, we have that $(d,2,r)$ is BN stable for $d\geq b_2(r)$. Since $$\rho(b_2(r),2,r)=2-(r+1)(2-b_2(r)+r)=(r+1)b_2(r)-r(r+3),$$ the formula in Lemma \ref{249} gives
\[
q_0=\left \lceil \frac{2(r+1)b_2(r)+3r^2-13r-2}{2(r-1)}\right \rceil.
\]
So Lemma \ref{249} implies semistability or stability if $g$ satisfies the inequality (\ref{277}) or (\ref{279}) respectively.
\end{proof}
For smaller $g$ we can obtain stability of the triple $(d,g,r)$ for $d$ large, per the following.
\begin{proposition} \label{prop-gSmall}
For a positive integer $b$, let $a_b$ be the least positive integer with $b+a_b=c(r-1)$ and $b\leq \min(a_b, \binom{c}{2})$. Then, if $g=1+b+qr(r-1)$ and $(d,g,r)$ is a BN triple with $d\geq r+1+a_b+q(r-1)^2$, then $(d,g,r)$ is BN semistable. If $g=2+b+qr(r-1)$ and $(d,g,r)$ is a BN triple with $d\geq b_2(r)+a_b+q(r-1)^2$, then $(d,g,r)$ is BN stable.
\end{proposition}
\begin{proof}
The semistability result is immediate from Remark \ref{rem-num} (2), (3), and (4), while the stability result is immediate from Remark \ref{rem-num} (5), (3), and (4).
\end{proof}
Given an integer $c\geq 1$, if we have an integer $b$ with $\binom{c-1}{2}+1\leq b \leq \binom{c}{2}$, then $a_b$ is the greater of $c(r-1)-b$ or the least integer at least $b$ such that $a_b+b$ is divisible by $r-1$. Hence, $a_b-b$ is bounded above by $\max(c(r-1)-2b, r-2)$. Minimizing $b$ in the range gives that $a_b-b\leq \max(c(r+2-c)-4, r-2)$. The former is maximized when $c=\frac{r+2}{2}$, so for all positive integers $b$, $a_b-b$ is bounded above by $\frac{r^2}{4}+r-3$. We also note the linear bound $a_b\leq b(2r-3)$.
Then the following corollary is immediate from Corollaries \ref{cor-sse}, \ref{cor-se} and Proposition \ref{prop-gSmall}.
\begin{corollary}\label{cor-gSmall}
Let $u = \left \lfloor \frac{(r+1)^2}{2(r-1)} \right \rfloor$ and let $k$ be the integer such that $(k-1)u < g-1 \leq ku$. Then
$(d,g,r)$ is BN semistable if $g\geq 1$ and
\[
d\geq \min\left(g+\frac{r^2}{4}+2r-3, (k+1)(r+1), (g-1)(2r-3)+r+1\right)
\]
and is BN stable if $g\geq 2$ and
\[
d\geq b_2(r) + \min\left(g+\frac{r^2}{4}+r-5, k(r+1),  (g-1)(2r-3)\right).
\]
\end{corollary}

We observe that Corollary \ref{cor-gSmall} yields the statements of Theorems \ref{largeMain} and \ref{thm-introstab}.

\subsection{Curves in \texorpdfstring{$\PP^4$}{P4}}\label{sec-p4}
In this subsection, we discuss the (semi)stability of the normal bundle of BN-curves in $\PP^4$ and prove Theorem \ref{theorem-P4}. As noted in \S \ref{sec-introunstable}, the triple $(6,2,4)$ is unstable. Excepting that case, we have the following result, which immediately implies Theorem \ref{theorem-P4}.
\begin{theorem}
Assume that the characteristic is not 2.
If $(d,g,4)$ is a BN triple,  $g\geq 1$,  $(d,g)\neq (6,2)$ and $(d,g)$ is not one of the 48 pairs listed in Table \ref{table-P4ss}, then $(d,g,4)$ is BN semistable. If, in addition, $g\geq 2$ and $(d,g)$ is not $(8,5)$ or one of the 63 additional pairs listed in Table \ref{table-P4s}, then $(d,g,4)$ is BN stable.
\end{theorem}

\begin{proof}
The theorem follows from the facts listed in Remark \ref{rem-num}.   For a given pair $(d,g)$ with $\rho(d,g,4)\geq 0$,  there are four reasons we apply to argue that $(d,g,4)$ is BN semistable.
\begin{enumerate}
\item Degeneration: $C$ degenerates to the union of a genus 1 curve, a collection of 1 and 2-secant lines, and a collection of triples of 4-secant rational cubics with each triple in a hyperplane.
\item Interpolation: if $2d+2g-2$ is divisible by 3, then $(d,g,4)$ is BN semistable (c.f.  Remark \ref{rem-num} (1)).
\item Configurations of elliptic curves: We specialize $C$ to a union of elliptic curves and use Remark \ref{rem-num} (6).
\item $(7,2)$ and $(16,14)$ are BN stable by Theorem \ref{724}. The assumption that the characteristic is not 2 is only needed for these two cases.
\end{enumerate}
Applying Remark \ref{rem-num} (3) allows us to conclude that $(d,g,4)$ is BN (semi)stable if $(d-9,g-12,4)$ is BN (semi)stable. In particular, all triples that could fail to be (semi)stable of genus at least 2 are of the form $(d_0+9k, g_0+12k, 4)$ for some $2\leq g_0\leq 13$ and some $d_0$.  Given such a bad $(d_0,g_0)$, we have $\rho(d_0+9k, g_0+12k, 4)=\rho(d_0,g_0,4)-3k$; since this must be nonnegative, we have that each $(d_0,g_0,4)$ that our methods fail to show BN (semi)stable  gives rise to only finitely many triples.

\begin{table}
    \centering
    \begin{tabular}{|c|c|c|c|}
      \hline
$g$ & $d_{min}$& Semistable threshold & Stable threshold\\
\hline
    2&6&7&9\\
    3&7&10&11\\
    4&8&10&10\\
    5&8&10&10\\
    6&9&12&12\\
    7&10&11&15\\
    8&11&13&15\\
    9&12&15&15\\
    10&12&14&16\\
    11&13&16&16\\
    12&14&18&18\\
    13&15&17&19\\
\hline
\end{tabular}
    \caption{Lowest $d$ for which semistable/stable degeneration argument applies for $2\leq g\leq 13$. $d_{min}$ is the least $d$ such that $\rho(d,g,4)$ is nonnegative.}
    \label{table-degenThresholds}
\end{table}

\begin{table}
    \centering
    \begin{tabular}{|c|c|c|c|}
    \hline
         $(d,g)$  semistability unknown&$k$ range & $(d,g)$  semistability unknown&$k$ range\\
         \hline
$(9k+8, 12k+3)$& $0\leq k\leq 2$&  $(9k+14, 12k+9)$ & $0 \leq k \leq 4$ \\
$(9k+9, 12k+3)$& $0\leq k\leq 4$ & $(9k+13, 12k+10)$ & $0\leq k \leq 1$ \\
 $(9k+8, 12k+4)$& $0\leq k\leq 1$ & $(13,11)$ & \\
$(9k+9, 12k+5)$& $0\leq k\leq 1$ & $(9k+15, 12k+11)$ & $0\leq k \leq 3$ \\
$(9, 6)$&& $(14,12)$&\\
     $(9k+11, 12k+6)$& $0\leq k\leq 3$&$(9k+15, 12k+12)$ & $0 \leq k \leq 2$ \\
$(10, 7)$& & $(9k+17, 12k+12)$ & $0 \leq k \leq 5$\\
$(9k+12, 12k+8)$& $0\leq k\leq 2$ & $(9k+16, 12k+13)$ & $0 \leq k \leq 2$\\
$(9k+12, 9k+9)$ & $0\leq k \leq 1$ & &  \\
             \hline
    \end{tabular}
    \caption{Pairs $(d,g)$ for which semistability of $(d,g,4)$ is unknown with $g\geq 2$. $(6,2)$ is excluded. There are 48 cases.}
    \label{table-P4ss}
\end{table}

\begin{table}
    \centering
    \begin{tabular}{|c|c|c|c|}
    \hline
         $(d,g)$ with stability unknown&$k$ range& $(d,g)$ with stability unknown&$k$ range\\
         \hline
$(9k+8, 12k+2)$&$0\leq k\leq 4$ & $(9k+14, 12k+8)$&$0\leq k\leq 6$\\
$(9k+7, 12k+3)$&$0\leq k\leq 1$ &  $(9k+13, 12k+9)$ & $0\leq k\leq 3$\\
$(9k+10, 12k+3)$&$0\leq k\leq 6$ & $(12,10)$ &  \\
$(9k+9, 12k+4)$&$0\leq k\leq 3$ & $(9k+15, 12k+10)$ & $0\leq k\leq 5$\\
$(9k+14, 12k+11)$ & $0\leq k\leq 2$ &
$(9k+10, 12k+6)$&$0\leq k\leq 2$  \\ 
$(9k+16, 12k+12)$ & $0\leq k\leq 4$ & 
$(9k+12, 12k+7)$&$0\leq k\leq 4$ \\ $(9k+15, 12k+13)$ & $0\leq k\leq 1$ & 
$(9k+11, 12k+8)$&$0\leq k\leq 1$ \\ $(9k+18, 12k+13)$ & $0\leq k\leq 6$ & & \\
\hline

    \end{tabular}
    \caption{Pairs $(d,g)$ for which stability of $(d,g,4)$ is unknown but semistability is known with $g\geq 2$. The case of genus 5 canonical curves $(8,5)$ is excluded because in that case the normal bundle is semistable but not stable. There are 63 additional cases.}
    \label{table-P4s}
\end{table}
When $2\leq g\leq 13$, the thresholds above which Remark \ref{rem-num} (2), (4), (5) and (6) guarantee stability or semistability are listed in Table \ref{table-degenThresholds}.  

When $g=2$, by Corollary \ref{cor-g2} and Theorem \ref{724}, $(d,2,4)$ is semistable for $d \geq 7$ and in fact $(d, 2, 4)$ is stable for $d=7, 9$. In addition, $b_2(4) = 10$. Hence, $(d, 2, 4)$ is stable if $d \geq 9$. Since $(6,2,4)$ is unstable, this only leaves the stability of $(8,2,4)$ unresolved.

When $g=3$, by Remark \ref{rem-num} (6), we can take the union of an elliptic normal curve and a two-secant elliptic curve of degree $d-5$ to deduce the semistability of $(d,6,4)$ for $d \geq 10$. However, since $3|12$, to deduce the stability we need to take an elliptic curve of degree $6$ and a 2-secant elliptic curve of degree $d \geq 5$. Hence, we can only deduce the stability of $(d, 3, 4)$ for $d \geq 11$.

When $g=6$, by Remark \ref{rem-num} (6), we can take the union  of two $5$-secant elliptic curves of degree $6$ and $d-6$ provided that $d-6 \geq 6$ to see the stability of $(d, 6, 4)$ for $d \geq 12$. 

When $r=4$, $u=\lfloor \frac{(r+1)^2}{2(r-1)} \rfloor = 4$. Hence, when $g=1 + 4k$, we can take $a_1 = \cdots = a_{k} = 4$  and $e_1 = \dots = e_k = 5$ in Remark \ref{rem-num} (6). Since $3$ and $14$ are relatively prime, we conclude that $(d, 1+4k, 4)$ is stable for $d \geq 5(k+1)$ provided that $k\geq 1$. This gives the stability threshold for $g=5$ and 9. Making one of the curves one or two  less secant, gives the stability threshold for $g=4,7,8$.

Repeatedly applying Remark \ref{rem-num} (4), we can further reduce the (semi)stability threshold. For example, since $(d,2,4)$ is semistable for $d \geq 7$, it follows that $(d,8,4)$ is semistable for $d \geq 13$ by taking $a=b=6$ and $c=4$. Similarly, $(d,2,4)$ is stable for $d \geq 9$, hence $(d, 8, 4)$ is stable for $d \geq 15$.

This process reduces the set of pairs $(d,g)$ such that $(d,g,4)$ could fail to be semistable to $(6,2)$ and the triples listed in Table \ref{table-P4ss}.

For a triple to be BN semistable but not BN stable, the rank and degree of the associated $N_{C\vert \PP^r}$ must not be coprime, that is, 3 must divide $2d+2g-2$. Hence, the triples we know to be semistable but do not know to be stable are exactly those with $2d+2g-2$ divisible by 3 that fall below the stability threshold we get by Remark \ref{rem-num} (2), (4), (5) and (6). These are exactly the  triples listed in Table \ref{table-P4s}.

\end{proof}

\section{Genus 2 degree 7 curve in  \texorpdfstring{$\PP^4$}{P4}}\label{724Section}
Since $(6,2,4)$ is an unstable triple (see \S \ref{sec-introunstable}),  $(d,g)=(7,2)$ is the smallest degree and genus  in $\PP^4$ for which the stability of the normal bundles of curves is an interesting question.  In this section, we use a cohomological approach to settle this case.
\begin{theorem}\label{724}
Assume the characteristic of the base field is not 2. 
If $C$ is a general BN-curve of genus 2 and degree 7 in $\PP^4$, then $N_{C\vert \PP^4}$ is stable.
\end{theorem}
\begin{proof}
Let $C\subset \PP^4$ be a general BN-curve of genus 2 and degree 7. Let $H$ be the hyperplane class on $C$, and let $K$ be the canonical class on $C$. Let $s_1,s_2$ be two independent global sections of $K$. Let $D=H-2K$ and let $t_1,t_2$ be two independent sections of $D$. Since $C$ is a general BN-curve, the line bundle $\OO_C(H)$ is a general line bundle of degree 7 on $C$. Hence, $D-K= H-3K$ is not effective, so $\lvert D\rvert$ is a basepoint free pencil. 

We first claim that the multiplication map
\[
M:H^0(C,\OO(2K))\otimes H^0(C,\OO(D))\ra H^0(C,\OO(H))
\]
is an isomorphism. Since $h^0(C,\OO(2K))=3$, $h^0(C,\OO(D))=2$, and $h^0(C,\OO(H))=6$ by Riemann-Roch, it suffices to show that $M$ has a six-dimensional image. If $q\in C$ is a ramification point of the canonical map $C\ra \PP^1$, then sections of $H^0(C,2K)$ vanish at $q$ to order $0,2$, or $4$. By the generality of $H$ and hence $D$, sections of $H^0(C,D)$ vanish to order 0 or 1 at $q$. Hence, $H^0(C,2K)\otimes H^0(C,D)$ contains pairs of sections whose  product vanishes at $q$ to order $i$ for all $0\leq i\leq 5$. Therefore, $M$ has a 6-dimensional image and   is an isomorphism.

We thus get an embedding $i:C\ra \PP^5$ given in coordinates by
\[
(s_1^2t_1, s_1s_2t_1,s_2^2t_1,s_1^2t_2,s_1s_2t_2,s_2^2t_2)
\]
Let $\tilde{C}=i(C)\subset \PP^5$, and let $x_0,\ldots,x_5$ be the coordinates on $\PP^5$. Since $C\subset \PP^4$  is embedded by an incomplete linear series contained in $\lvert H \rvert$, it is the projection of the curve $\tilde{C}\subset \PP^5$ from a general point $p\in \PP^5$. Let $\pi_p:\PP^5\dashrightarrow \PP^4$ denote the projection, so $\pi_p(\tilde{C})=C$.

The pair of maps $C \to \PP^1$ defined by $|K|$ and $|D|$ embed $C$  into $\PP^1 \times \PP^1$ as a curve of type $(2,3)$.
Since $D= H-2K$, the line bundle $\OO_{\PP^1\times \PP^1}(1,2)$  restricts to $H$ on  $C$. Since $H^1(\OO_{\PP^1 \times \PP^1}(-1,-1))=0$, the long exact sequence associated to the exact sequence
$$0 \to \OO_{\PP^1 \times \PP^1}(-1,-1) \to \OO_{\PP^1 \times \PP^1}(1,2) \to \OO_C(H) \to 0$$ shows that the linear system $|\OO_{\PP^1 \times \PP^1}(1,2)|$ restricts to the complete linear system $|H|$ on $C$.  Hence, $\tilde{C}$ is contained in the scroll $\tilde{S}$ given by the embedding of $\PP^1 \times \PP^1$ by the complete linear series $|\OO_{\PP^1 \times \PP^1}(1,2)|$.  The normal bundle $N_{\tilde{C}\vert \tilde{S}}\cong \OO_{\tilde{S}}(\tilde{C})\vert_{\tilde{C}}$  has degree 12. 

Let $S$ be $\pi_p(\tilde{S})$ in $\PP^4$. Since $\tilde{S}$ contains $\tilde{C}$, $S$ contains $C$. Moreover, since $p$ is general, $p$ is not contained in the closure of the union of lines between a point on $\tilde{C}$ and a distinct point on $\tilde{S}$. So $S$ is smooth along $C$, and the pullback by projection induces an isomorphism $N_{C\vert S}\cong N_{\tilde{C}\vert \tilde{S}}$. Hence $N_{C\vert S}$ is a subbundle of $N_{C\vert \PP^4}$ isomorphic to the degree 12 bundle $\OO_C(2D+3K)\cong \OO_C(2H-K)$, and we have a sequence
\[
0\ra \OO_C(2H-K)\xrightarrow{i} N_{C\vert \PP^4}\ra N_{S\vert \PP^4}\vert_C\ra 0.
\]
Since $\mu(\OO_C(2H-K))=12$ and $\mu(N_{C\vert \PP^4})=12 \frac{1}{3}$ are consecutive Farey fractions, to prove $N_{C\vert \PP^4}$ stable, by \cite{clv23} Lemma 2.5 it suffices to show:
\begin{enumerate}
    \item The inclusion $i:\OO_C(2H-K)\ra N_{C\vert \PP^4}$ does not split.
    \item   $N_{S\vert \PP^4}\vert_C$ is stable.
\end{enumerate}
These are the content of Lemma \ref{noSplitting} and Lemma \ref{stableQuotient} below.
\end{proof}
\begin{lemma} \label{noSplitting}
Under the hypotheses of Theorem \ref{724}, $H^0(C,N^*_{C\vert \PP^4}(2H-K))=0$.
\end{lemma}
\begin{proof}
We prove this result by studying $H^0(C, N^*_{\tilde{C}\vert \PP^5}(2H-K))$. The conormal bundle fits into the exact sequence
\[
0\ra N^*_{\tilde{C}\vert \PP^5}(2H-K)\ra T^*\PP^5\vert_{\tilde{C}}(2H-K)\ra T^*\tilde{C}(2H-K)\ra 0.
\]
For $i\in\{0,1,3,4\}$, define the sections $u_i=x_is_1d\frac{x_{i+1}}{x_i}\in H^0(\tilde{S}, T^*\PP^5\vert_{\tilde{S}}\otimes \OO_{\tilde{S}}(3,1))$, and for $i\in \{0,1,2\}$, define $v_i=x_it_1d\frac{x_{i+3}}{x_i}\in H^0(\tilde{S}, T^*\PP^5\vert_{\tilde{S}}\otimes \OO_{\PP^1\times \PP^1}(2,2))$. We have that
 $T^*\PP^5\vert_{\tilde{S}}\otimes \OO_{\PP^1\times \PP^1}(3,2)$ is globally generated, with global sections
\begin{align*}
t_ju_{i} &\;\text{for}\; i\in \{0,1,3,4\},\, j\in \{1,2\}\\
s_jv_i& \; \text{for}\; i\in\{0,1,2\},\, j\in \{1,2\}.
\end{align*}
These span a 12-dimensional subspace of the 13-dimensional space $H^0(\tilde{C},T^*\PP^5\vert_C(2H-K))$. Their image in $H^0(\tilde{C}, T^*\tilde{C}(2H-K))$ is the vector space given by
\begin{equation}\label{257}
\langle s_1t_1^2,s_1t_1t_2,s_1t_2^2,s_2t_1^2,s_2t_1t_2,s_2t_2^2\rangle\cdot   s_1^2d\frac{s_2}{s_1}+\langle s_1^3,s_1^2s_2,s_1s_2^2,s_2^3\rangle \cdot t_1^2 d\frac{t_2}{t_1}
\end{equation}
Every form in the first part of the sum vanishes on the ramification divisor associated to $K$, $R_K$, while no form in the second part of the sum does (since $t_1^2d\frac{t_2}{t_1}$ vanishes on the ramification divisor of $D$, which is disjoint from that of $K$, and no section in the space $\langle s_1^3,s_1^2s_2,s_1s_2^2,s_2^3\rangle$ can vanish at more than three of the points of $R_K$). Hence, the vector space (\ref{257}) has dimension 10.

We then have that the two forms $R_1:=t_2u_0-t_1u_3$ and $R_2:=t_2u_1-t_1u_2$ span the kernel of the map $H^0(\tilde{S}, T^*\PP^5\vert_{\tilde{S}}(3,2))\ra H^0(\tilde{C},T^*\tilde{ C}(2H-K))$. Their images in $H^0(\tilde{C},\OO(H-K)^{\oplus 6})$ are given by
\begin{align*}
R'_1&=s_2t_2dx_0-s_1t_2dx_1-s_2t_1dx_3+s_1t_1dx_4 \\
R'_2&=s_2t_2dx_1-s_1t_2dx_2-s_2t_1dx_4+s_1t_1dx_5.
\end{align*}
Let $V\subset H^0(\tilde{C}, \OO_{\tilde{C}}(H-K)^{\oplus 6})$ be the image of $H^0(N^*_{\tilde{C}\vert \PP^5}(2H-K))$. We have $2\leq \dim(V)\leq 3$ because $R'_1,R'_2$ are contained in $V$, $h^0(T^*\PP^5\vert_{\tilde{C}}(2H-K))=13$, and 
\[
\dim \mathrm{Im}( H^0(\tilde{C}, T^*\PP^5\vert_{\tilde{C}}(2H-K))\ra H^0(\tilde{C}, T^*\tilde{C}(2H-K)))\geq 10.
\]
For the lemma to be false, we must have that for a general map $$e_p:\OO_{\tilde{C}}(H-K)^{\oplus 6}\ra \OO_{\tilde{C}}(H-K),$$ the intersection $\ker(e_p)\cap V\subset H^0(\tilde{C},\OO_{\tilde{C}}(H-K)^{\oplus 6})$ is nonzero. If $\dim(V)=2$, this does not occur: if $e_p$ is induced by projection from $(0,1,0,0,0,0)$---and hence sends $dx_1$ to 1 and all other $dx_i$ to 0---we have $e_p(R'_1)=-s_1t_2$ and $e_p(R'_2)=s_2t_2$. Since these two images are linearly independent, $e_p$ is nonzero on every nonzero element of $V$. 

Now suppose $\dim(V)=3$, and $V$ contains a third linearly independent section $w$. More specifically, by subtracting appropriate constant multiples of $R'_1$ and $R'_2$, let $w$ be that section whose $dx_1$ term is a multiple of $t_1$. Suppose the lemma is false, so $\ker(e_p)\cap V\neq 0$ for all maps $e_p$. Again using projection from $(0,1,0,0,0,0)$, we have that since the $dx_1$ term in $w$ is a multiple of $t_1$, it is actually 0. Say we have
\[
w=\sum_{i\neq 1} f_i dx_i.
\]
Projecting from $(1,1,0,0,0,0)$ then gives that $f_0$ is a multiple of $t_2$, and projecting from $(0,1,1,0,0,0)$ gives that $f_2$ is a multiple of $t_2$.

Let $\tilde{w}\in H^0(\tilde{S}, \OO(1,1)^{\oplus 6})$ be the section restricting to $w$ by the isomorphism  $$H^0(\tilde{S},\OO(1,1)^{\oplus 6})\to H^0(\tilde{C},\OO(H-K)^{\oplus 6}),$$ and let $f:H^0(\tilde{S},\OO(1,1)^{\oplus 6})\to H^0(\tilde{S}, \OO(3,2))$ be the map coming from restricting the Euler sequence on $\PP^5$ to $\tilde{S}$.

Because $\tilde{w}$ is not in the span of $R_1$ and $R_2$, $\tilde{w}$ does not lift to a global section of $T^*\PP^5\vert_{\tilde{S}}(3,2)$, so $f(\tilde{w})\neq 0$. And we have that $f(\tilde{w})$ is actually a multiple of $t_2$, because $f_0,f_1,f_2$ are all multiples of $t_2$, while $dx_3,dx_4,dx_5$ all map to a multiple of $t_2$, so $f(\tilde{w})=t_2g$ for some $g$. Finally, since $w$ lifts to a section of $T^*\PP^5\vert_{\tilde{C}}(2H-K)$, $f(\tilde{w})$ maps to 0 under the map $H^0(\tilde{S}, \OO(3,2))\ra H^0(\tilde{C}, 2H-K)$, and hence $f(\tilde{w})$ is a constant multiple of the section of $\OO(3,2)$ cutting out $\tilde{C}$. This contradicts $f(\tilde{w})$ being a nonzero multiple of $t_2$.
\end{proof}
\begin{lemma}\label{stableQuotient}
In the same setup as Theorem \ref{724}, $N_{S\vert \PP^4}\vert_C$ is stable.
\end{lemma}
\begin{proof}
Suppose  $N_{S\vert \PP^4}\vert_C$ is unstable for $C$ a general curve on $S$. We  continue to regard $C,S$ as the projection of $\tilde{C}\subset \tilde{S}\subset \PP^5$ from a general point $p$ to $\PP^4$ via the projection map $\pi_p$. If $C'\subset \tilde{S}$ is any smooth curve of class $(3,2)$, then $N_{S\vert \PP^4}\vert_{\pi_p(C')}$ is likewise unstable, and has some maximal destabilizing line quotient bundle $L_{C'}$ of degree at most 12 . The map $C'\mapsto L_{C'}$ is a rational section of the relative Picard variety associated to the moduli of $(3,2)$ curves on $\tilde{S}$; in particular, by \cite{Woo14}, 
$L^*(C')\cong \OO_{\tilde{S}}(a,b)\vert_{C'}$ for some integers $a,b$ that do not depend on $C'$ or $p$, given that $p$ is general. 

We establish a contradiction by showing $H^0(C,N^*_{S\vert \PP^4}\vert_C(L))=0$, where $L\cong \OO_{\tilde{S}}(a,b)\vert_C$.  
This proof has three main steps, again mostly working in $\PP^5$:
\begin{enumerate}
    \item If $\deg(L)\leq 12$, then $h^0(\tilde{C}, N^*_{\tilde{S}\vert \PP^5}\vert_{\tilde{C}}(L))\leq 3$.
    \item Given a nonzero section $w\in H^0(\tilde{C},N^*_{\tilde{S}\vert \PP^5}\vert_{\tilde{C}}(L))$, the image of $w$ in $H^0(\tilde{C}, L\otimes \OO_{\tilde{C}}(-H)^{\oplus 6})$ spans a subspace of $H^0(C, L\otimes \OO_{\tilde{C}}(-H)))$ of dimension at least three.
    \item Facts (1) and (2) imply $H^0(C, N^*_{S\vert \PP^4}\vert_C(L))=0$ for $C,S$ the projection of $\tilde{C},\tilde{S}$ from a general point.
\end{enumerate}

\textbf{(1)  $h^0(C, N^*_{\tilde{S}\vert \PP^5}\vert_{\tilde{C}}(L))\leq 3$:}

We define a divisor class associated to $L$ as follows. If $L\cong \OO_{\tilde{C}}(D+4K)$, set $L'=D+4K$. Otherwise,
let $\Delta$ be a general effective divisor class of degree $12-\deg(L)$ and let $L'$ be the divisor class associated to the line bundle $L(\Delta)$. We show $h^0(\tilde{C},N^*_{\tilde{S}\vert \PP^5}\vert_{\tilde{C}}(L'))\leq 3 $, which immediately implies the claim.
Applying the Euler sequences to $T^*\PP^5$ and $T^*\tilde{S}$ gives the commutative diagram on $\tilde{C}$,
\[
\begin{tikzcd}
 N^*_{\tilde{S}\vert \PP^5}\vert_{C}(L')\arrow[r, hookrightarrow]\arrow[d, "\cong"]& T^*\PP^5\vert_{C}(L')\arrow[r, two heads]\arrow[d, hookrightarrow]& T^*\tilde{S}\vert_{C}(L') \arrow[d,hookrightarrow]\\
  N^*_{\tilde{S}\vert \PP^5}\vert_{C}(L')\arrow[r, hookrightarrow]&\OO_{C}(L'-H)^{\oplus 6}\arrow[r,"M_1"]\arrow[d,two heads]&\OO_{C}(L'-K)^{\oplus 2}\oplus \OO_{C}(L'-D)^{\oplus 2}\arrow[d,two heads]&\\
 &\OO_{C}(L')\arrow[r] &\OO_{C}(L')^{\oplus 2}
\end{tikzcd},
\]
which in turn gives an exact sequence
\[
0\ra N^*_{\tilde{C}\vert \PP^5}(L')\ra \OO_{\tilde{C}}(L'-H)^{\oplus 6}\xrightarrow{M_1} \OO_{\tilde{C}}(L'-K)^{\oplus 2}\oplus \OO_{\tilde{C}}(L'-D)^2\xrightarrow{M_2} \OO_{\tilde{C}}(L')\ra 0.
\]
We can describe $M_1$ explicitly as a matrix of partial derivatives of the $x_i$, where the $i$th column consists of the four partial derivatives $\frac{\partial x_i}{\partial s_1}, \frac{\partial x_i}{\partial s_2},\frac{\partial x_i}{\partial t_1},\frac{\partial x_i}{\partial t_2}$.

In that framing, we have
\[
M_1=\begin{bmatrix}
2s_1t_1&s_2t_1&0&2s_1t_2&s_2t_2&0\\
0&s_1t_1&2s_2t_1&0&s_1t_2&2s_2t_2\\
s_1^2&s_1s_2&s_2^2&0&0&0\\
0&0&0&s_1^2&s_1s_2&s_2^2
\end{bmatrix}
\]
and 
\[
M_2=\begin{bmatrix}
s_1&s_2&-2t_1&-2t_2
\end{bmatrix}
\]
First suppose $L'=4K+D$.
Then $h^0(T^*\PP^5\vert_{\tilde{C}}(L'))=8$, and if we set
\[
u_i=x_is_1d\frac{x_{i+1}}{x_i}\in H^0(T^*\PP^5\vert_{\tilde{C}}(3K+D)),
\]
for $i\in \{0,1,3,4\}$, we have that the 8 sections $\{s_ju_i\}_{j\in\{1,2\},i\in \{0,1,3,4\}}$ are linearly independent and hence span $H^0(T^*\PP^5\vert_{\tilde{C}}(L'))$. We then have
\[
M_1\begin{bmatrix}
u_0&u_1&u_3&u_4
\end{bmatrix}
=
\begin{bmatrix}
-s_1s_2t_1 & -s_2^2t_1 & -s_1s_2t_2 & -s_2^2t_2\\
s_1^2t_1 & s_1s_2 t_1 & s_1^2 t_2 & s_1s_2t_2\\
0&0&0&0\\
0&0&0&0
\end{bmatrix}
\]
Hence, the images of the $s_ju_i$ span $H^0(\tilde{C},D+2K)\cdot \begin{bmatrix}s_2\\-s_1\\0\\0\end{bmatrix}$, a six-dimensional space, and $h^0(\tilde{C}, N^*_{\tilde{S}\vert \PP^5}\vert_{\tilde{C}}(L'))=2$.

Now suppose $L'\neq 4K+D$, so $\deg(L')=12$. We have that $L'-D-3K$ is a degree 3 divisor, and know $L'-D-4K$ is \emph{not} effective, 
so $L'-D-3K$ is the class of a base-point free pencil. Let $\beta_1,\beta_2$ be two independent sections of it. Likewise, let $\gamma$ be a section of $L'-2D-2K$. Then the images of the sections $\beta_js_i\in H^0(\tilde{C},T^*\PP^5\vert_{\tilde{C}}(L'))$  under $M_1$ span $H^0(\tilde{C}, L'-3K) \cdot \begin{bmatrix}s_2\\-s_1\\0\\0\end{bmatrix}$, a seven-dimensional space. Set $v_i=x_it_1d\frac{x_{i+3}}{x_i}$. The images of $\gamma v_i$ under $M_1$ span $H^0(\tilde{C}, 2K) \cdot \begin{bmatrix}0\\0\\\gamma t_2\\ -\gamma t_1\end{bmatrix}$, a three dimensional space. So the image of $H^0(\tilde{C}, T^*\PP^5\vert_{\tilde{C}}(L'))$ under $M_1$ has dimension at least 10; since 
\begin{align*}
h^0(\tilde{C}, T^*\PP^5\vert_{\tilde{C}}(L'))&=h^0(\tilde{C},\OO_{\tilde{C}}(L'-H)^{\oplus 6})-h^0(\tilde{C},\OO_{\tilde{C}}(L'))\\
&=24-11=13,
\end{align*}
we have that $h^0(N^*_{\tilde{S}\vert \PP^5}\vert_{\tilde{C}}(L'))\leq 3$.\\
\\
\textbf{(2) Given a nonzero section $w\in H^0(N^*_{\tilde{S}\vert \PP^5}\vert_C(L))$, the image of $w$ in $H^0(\tilde{C}, L\otimes \OO_{\tilde{C}}(-H)^{\oplus 6}))$ spans a subspace of $H^0(\tilde{C}, L\otimes \OO_{\tilde{C}}(-H)))$ of dimension at least three:}
For this part, we choose $s_1$ and $t_1$ both vanishing to order 1 at some point $q\in \tilde{C}$; hence neither $s_2$ nor $t_2$ vanish at $q$. Let the image of $w$ in $H^0(\tilde{C},L-H)^{\oplus 6}$ be 
$\sum_i w_i dx_i$. We need to show that the $w_i$ span a subspace of $H^0(\tilde{C},L-H)$ of dimension at least 3.
Let $e_i$ be the order of vanishing of $w_i$ at $q$; $e_i$ is either a nonnegative integer or $\infty$. Then the orders of vanishing at $p$ of the $i$th row of $M_1(\sum_i w_i dx_i) $ fit into the $4\times 6 $ matrix
\begin{equation}\label{tropicalMx}
\begin{bmatrix}
2+e_0&1+e_1&\infty&1+e_3&e_4&\infty\\
\infty&2+e_1&1+e_2&\infty&1+e_4&e_5\\
2+e_0&1+e_1&e_2&\infty&\infty&\infty\\
\infty&\infty&\infty&2+e_3&1+e_4&e_5
\end{bmatrix}.
\end{equation}
For $M_1(w)$ to vanish, each of its four components must vanish to infinite order at $q$. For the requisite cancellation to occur, we must have that the minimum value of each row of (\ref{tropicalMx}) occurs at least twice in that row. We claim that this requirement implies that there are at least 3 distinct finite values among the $e_i$.  First, we note that none of $e_2,e_4,e_5$ can be minimal; as otherwise there could be no shared minimum value in row 3, 1, 2 respectively. If $e_0$ is minimal, without loss of generality we may suppose $e_0=0$. Then either $e_1=1$ or $e_2=2$ to cancel in the third row. If $e_1=1$, to cancel in the second row we need at least one of $e_2=2$, $e_4=2$, or $e_5=3$; in all three cases we have at least three distinct values of $e_i$. If $e_1\neq 1$, $e_2=2$, and from the second row either $e_4=2$ or $e_5=3$; in the latter case we immediately have three distinct values of the $e_i$, while in the former case we have that from the fourth row either $e_3=1$ or $e_5=3$, so three distinct values of the $e_i$ appear.

If $e_1=0$, then $e_2=1$ to match in the third row. Then row 1 shows $e_3=1$ or $e_4=2$; in the latter case, there are three distinct values of $e_i$, and in the former case,  to cancel the fourth row we must have $e_4=2$ or $e_5=3$, again giving three distinct values of $e_i$.

Finally, if $e_3=0$, then either $e_4=1$ or $e_5=2$ to match in the fourth row. If $e_4=1$, then to match in the second row we either have $e_5=2$ and have 3 distinct $e_i$, have $e_1=0$, a case covered above, or have $e_2=1$. In this last case, we must have $e_1=0$ to match in the third row, and there is no match to the 1 in the first row of that column.

Hence, there are at least 3 distinct values of $e_i$ represented, and the $w_j$ vanish to at least three different orders at $q$. They hence span a subspace of  $H^0(\tilde{C}, L-H)$ of dimension at least 3. Note that due to the coefficients of 2 in the matrix $M_1$ this argument fails in characteristic 2. \\
\\
\textbf{(3) $H^0(C,N^*_{S\vert \PP^4}(L)\vert_C)=0$:}
As in the proof of Lemma \ref{noSplitting}
, for $p=(p_0,p_1,p_2,p_3,p_4,p_5)\in \PP^5\setminus \tilde{C}$ there is a diagram
\[
\begin{tikzcd}
N^*_{S\vert \PP^4}(L)\vert_C\arrow[r, hookrightarrow]\arrow[d]&N^*_{\tilde{S}\vert \PP^5}(L)\vert_C \arrow[r, two heads] \arrow[d]&N^*_{\tilde{C}\to p}(L)\arrow[d,"\cong"]\\
\OO_C(L-H)^{\oplus 5} \arrow[r,hookrightarrow]& \OO_C(L-H)^{\oplus 6} \arrow[r, two heads, "e_p"] &\OO_C(L-H),
\end{tikzcd}
\]
where $e_p$ is given by $e_p(\sum_i f_idx_i)=\sum_i p_if_i$.

A section $w\in H^0(C,N^*_{\tilde{S}\vert \PP^5}(L)\vert_C) $ lifts to  $H^0(N^*_{S\vert \PP^4}(L)\vert_C)$ if and only if its image in $H^0(C,\OO_C(L-H))$ is sent to zero by $e_p$. By (2), given any such $w$ nonzero, the codimension in $\Hom(\OO_C(L-H)^{\oplus 6},\OO_C(L-H))$, of $e_p$ with $e_p(w)=0$ is at least three. So if $p$, and hence $e_p$, is chosen in general position, the locus of $w\in H^0(C,N^*_{\tilde{S}\vert \PP^5}(L)\vert_C)$ with $e_p(w)=0$ is a subspace of codimension at least three. Hence, by (1), we have
\[
H^0(C,N^*_{S\vert \PP^4}\vert_C(L))=H^0(C, N^*_{\tilde{S}\vert \PP^5}\vert_C(L))\cap \ker(e_p) =0.
\]
\end{proof}

\bibliographystyle{plain}

\end{document}